\documentclass[10 pt]{amsart}

\usepackage{amsfonts, amssymb, amsmath, amsthm}
\usepackage[latin1]{inputenc}
\usepackage{amsmath}
\usepackage{amssymb}
\usepackage{amsfonts}
\usepackage{latexsym}
\usepackage{mathtools}
\usepackage{amscd}
\usepackage[mathscr]{euscript}
\usepackage{enumitem}
\usepackage{xy} \xyoption{all}
\usepackage{mathtools}
\usepackage{color}
\usepackage{soul}
\usepackage{comment}

\newcommand{\annotation}[1]{}

\newcommand{\Z}{\mathbb Z}

\newcommand{\aut}{\textnormal{Aut}}

\newcommand{\autinfty}{\textnormal{Aut}^{(\infty)}}

\newcommand{\inv}[1]{\textnormal{Inv}(\sigma_{n})}
\newcommand{\infinv}[1]{\textnormal{Inv}^{\infty}(\sigma_{n})}

\newcommand{\sym}{\textnormal{Sym}}

\newcommand{\fix}{\textnormal{Fix}}

\newcommand{\htop}{h_{\textrm{top}}}
\newcommand{\finorm}{\mathcal{I}_{N}}

\newcommand{\Aut}{\mathrm{Aut}}
\newcommand{\Ker}{\mathrm{ker}}

\newcommand{\supp}{\mathrm{supp}}
\newcommand{\Alt}{\mathrm{Alt}}
\newcommand{\Sym}{\mathrm{Sym}}
\newcommand{\PAut}{\mathrm{PAut}}
\newcommand{\gate}{\textrm{gate}}
\newcommand{\Homeo}{\textrm{Homeo}}

\newcommand{\ball}{\mathscr{B}}
\newcommand{\fins}{\textrm{Fin}}
\newcommand{\fol}{\textrm{F{\o}l}}

\newtheorem*{theorem*}{Theorem}
\newtheorem{theorem}{Theorem}

\newtheorem{corollary}[theorem]{Corollary}
\newtheorem{proposition}[theorem]{Proposition}
\newtheorem{lemma}[theorem]{Lemma} 
\theoremstyle{remark}
\newtheorem{remark}[theorem]{Remark}
\theoremstyle{definition}
\newtheorem{definition}[theorem]{Definition}
\newtheorem*{definition*}{Definition}

\newtheorem*{notation*}{Notation}

\input xy
\xyoption{all}
\author{Ville Salo}
\author{Scott Schmieding}

\thanks{SS was partially supported by the National Science Foundation grant DMS-2247553.}

\begin{document}
\title{Finitary Ryan's and local $\mathcal{Q}$ entropy for $\mathbb{Z}^{d}$ subshifts}

\maketitle
\begin{abstract}

For the action of a group $G$ by homeomorphisms on a space $X$, the automorphism group $\aut(X,G)$ consists of all self-homeomorphisms of $X$ which commute with $x \mapsto g \cdot x$ for every $g \in G$.

A theorem of Ryan shows that for a nontrivial irreducible $\mathbb{Z}$-shift of finite type $(X,\sigma_{X})$, the center of $\aut(X,\sigma_{X})$ is generated by the shift $\sigma_{X}$. A finitary version of this for $\mathbb{Z}$-shifts of finite type was proved by the second author for certain full shifts, and later generalized by Kopra to nontrivial irreducible $\mathbb{Z}$-shifts of finite type. We generalize these finitary Ryan's theorems to shifts of finite type over more general groups. In particular, we prove that for contractible $\mathbb{Z}^{d}$-shifts of finite type possessing a fixed point, there is
a finitely generated subgroup of the automorphism group whose centralizer in the group of homeomorphisms of the shift space is the subgroup of shifts. We also prove versions of this for full shifts over any infinite, finitely generated group on sufficiently nice alphabet sizes.

For the action of $G$ on $X$, the stabilized automorphism group $\autinfty(X,G)$ consists of all self-homeomorphisms of $X$ for which there is a finite index subgroup $H \subset G$ such that $\varphi$ commutes with the action of every element of $H$. Motivated by the study of stabilized automorphism groups for shifts of finite type, we introduce an entropy-like quantity for pointed groups which we call local $\mathcal{Q}$ entropy,
a generalization of a notion called local $\mathcal{P}$ entropy previously introduced by the first author. This quantity is invariant under pointed isomorphism. Using the finitary Ryan's theorems, we prove that the local $\mathcal{Q}$ entropy of the stabilized automorphism group of a nontrivial contractible $\mathbb{Z}^{d}$-shift of finite type recovers the topological entropy of the underlying shift system up to a rational multiple. We then use this to give a complete classification up to isomorphism of the stabilized automorphism groups of full shifts over $\mathbb{Z}^{d}$.
\end{abstract}
\section{Introduction}
Let $G$ be a group. By a $G$-system $(X,G)$ we mean a compact metric space $X$ equipped with a $G$-action by homeomorphisms of $X$. An automorphism of a $G$-system $(X,G)$ is a homeomorphism $\varphi \colon X \to X$ such that $\varphi(g \cdot x) = g \cdot \varphi(x)$ for all $g \in G$ and $x \in X$. The set of all automorphisms of $(X,G)$ forms a group under composition which we call the automorphism group of $(X,G)$ and denote by $\aut(X,G)$, or just $\aut(X)$ when the underlying group action is clear.

Here we are primarily concerned with automorphism groups of symbolic $G$-actions. Symbolic $G$-actions are given by $G$-subshifts, i.e. the action of $G$ by translation on compact $G$-invariant subsets of $\mathcal{A}^{G}$ where $\mathcal{A}$ is some finite set. The most fundamental such systems are the $G$-shifts of finite type (abbreviated herein as SFT), which are subshifts obtained by forbidding some finite set of patterns from appearing.

The automorphism groups of symbolic systems have been heavily studied, especially in the setting of shifts of finite type, starting with full $\mathbb{Z}$-shifts by Hedlund in~\cite{Hedlund1969} and later the mixing case over $\mathbb{Z}$ in~\cite{BLR} (see also~\cite[Sec. 7]{BSstablealgebra} for a brief survey). It is known \cite{BLR} that for a nontrivial mixing SFT over $\mathbb{Z}$, the automorphism group always contains free groups and every finite group, and it follows from~\cite[Theorem 3]{Hochman2010} that the same holds for positive entropy SFT's over $\mathbb{Z}^{d}$ having a dense set of minimal points. But in general, the algebraic structure of the group of automorphisms of a shift of finite type remains mysterious. Of note is the result of Ryan~\cite{Ryan1,Ryan2}, showing that for a mixing shift of finite type $(X,\sigma_{X})$ over $\mathbb{Z}$, the center of $\aut(X)$ is generated by the shift map $\sigma_{X}$. This was later extended by Hochman~\cite{Hochman2010}, showing that for topologically transitive $\mathbb{Z}^{d}$-shifts of finite type with positive entropy, the center of $\aut(X)$ is the subgroup $\mathbb{Z}^{d}$ of shift maps.

Later, the second author proved in~\cite{SaloFinitaryRyans2019} a finitary version of Ryan's theorem for the full $\mathbb{Z}$-shift $X_{4}$ on four symbols: there exists a finite set of automorphisms $F = \{f_{i}\} \subset \aut(X_{4})$ such that the centralizer of $F$ in $\aut(X_{4})$ consists of exactly the powers of the shift. Following this, Kopra extended this in~\cite{Kopra2020} to obtain a finitary Ryan's theorem for every infinite irreducible shifts of finite type over $\mathbb{Z}$, and also showed that one can always find a set $F$ containing only two automorphisms whose centralizer is precisely powers of the shift.

Our first main set of results here are a set of finitary Ryan's theorems for shifts of finite type over more general groups.

\begin{theorem}
\label{thm:introFinitaryRyanZd}
If $X$ is a contractible $\Z^d$-shift of finite type with a fixed point, then there is a finitely generated subgroup $H$ of the automorphism group of $X$ such that the centralizer of $H$ in $\Homeo(X)$ is the subgroup of shifts $\mathbb{Z}^{d}$. In particular, the centralizer of $H$ in $\aut(X)$ is the subgroup of shifts $\mathbb{Z}^{d}$.
\end{theorem}
The contractibility hypothesis was introduced in~\cite{PoirierSaloContractible}, as a strengthening of the notion of strong irrreducibility. Among other useful properties, for $\mathbb{Z}^{d}$-SFT's it was proved in~\cite{PoirierSaloContractible} that contractible SFT's coincide with the class of subshifts having map extension property of Meyerovitch introduced in~\cite{MeyerovitchEmbeddingTheorem}.

The above result gives a stronger conclusion than the previously proved finitary Ryan's theorems over $\mathbb{Z}$, in that it refers to the centralizer in all of $\Homeo(X)$; however, the hypothesis requires a fixed point, and the number of generators needed is more than the known results for $\mathbb{Z}$.

Over more general groups, we obtain an analogous strong finitary Ryan's theorem for full shifts on sufficiently nice alphabet sizes.
\begin{theorem}\label{thm:intropotentsstrongryansomgsomany}
\label{thm:FullShiftStrongRyan}
Let $G$ be infinite, finitely generated and let $(X, G)$ be a full shift on $n$ symbols where $n=m^{k}$ for some $m \ge 5, k \ge 4$. Then there is finitely generated subgroup $H$ of the automorphism group of $X$ such that the centralizer of $H$ in $\Homeo(X)$ is the subgroup of shifts $G$. In particular, the centralizer of $H$ in $\aut(X)$ is the subgroup of shifts $G$.
\end{theorem}

We note that for $G$ infinite, finitely generated and residually finite and any alphabet $\mathcal{A}$, there is always a finite index subgroup $K \subset G$ such that the system $(\mathcal{A}^{G},K)$ satisfies the above.

The proofs for the finitary Ryan's theorems are built on transitivity properties for the action of finitely generated subgroups of the automorphism group on homoclinic points of the shift, many of which are interesting in their own right. The finitary Ryan's theorems above also give an answer to Question 7.2 posed in~\cite{BC-VR-B2025}.

In~\cite{HKS2022}, a certain stabilization of the automorphism group of a system was introduced\footnote{In~\cite{HKS2022}, the definition was given for $\mathbb{Z}$-actions; a more general definition was given in~\cite{SaloGateLattices}.}. Given a system $(X,G)$, we define the stabilized automorphism group $\autinfty(X,G)$ to be the set of self-homeomorphisms $\varphi$ of $X$ such that there exists a finite index subgroup $H \subset G$ for which $\varphi$ commutes with the action of every element of $H$. Thus $\autinfty(X,G)$ is the union of automorphism groups $\aut(X,H)$ as $H$ ranges over all the finite index subgroups of $G$. The stabilized automorphism group by definition contains the original automorphism group of a system, but in general is larger.

Several questions which remain out of reach for the classical automorphism groups have been resolved in the stabilized setting. It was proved in~\cite{HKS2022} that for full shifts (and later significantly generalized in~\cite{SaloGateLattices}), the commutator subgroup of the stabilized automorphism group is an infinite simple group, and agrees with the commutator. It follows from this and~\cite{BoyleEventualExtensions} that for full shifts over $\mathbb{Z}$, the stabilized automorphism group is generated by partial shifts and elements of finite order (the nonstabilized version of this, called the Finite Order Generation Conjecture, remains open). It was also shown in~\cite{EpperleinFreeConjugacy} that finite order elements which act freely on a full shift must be conjugate in the stabilized group, and the results of~\cite{EpperleinSchmieding2024} show that isomorphisms between stabilized automorphism groups of full shifts are spatially induced at the level of the space of subsystems of the shifts.

The stabilized automorphism groups have also been studied in the setting of minimal subshifts, in~\cite{Jones-BaroToeplitz,EspinozaJones-Baro2024}.

In the context of subshifts over more general groups, less is known about the stabilized automorphism groups. In~\cite{SaloGateLattices}, the second author introduced a subgroup of the stabilized automorphism group called the gate lattice subgroup. The gate lattices are finite order automorphisms obtained, very roughly speaking, by translated copies a local permutation rule. They play an important role in the study of the stabilized groups, and serve as a good generalization\footnote{We have, at this point, no suitable generalization of Krieger's dimension group for general SFT's over groups other than $\mathbb{Z}$; since the inerts are defined as the automorphisms which act trivially on the dimension group, in the absence of such an object, there is no obvious extension of the definition of inertness to SFT's over general groups.} of the inert automorphisms of $\mathbb{Z}$-shift of finite type. In~\cite{SaloGateLattices}, it was shown that for any countably infinite residually finite group $G$ and shift of finite type over $G$ with the eventual filling property, the group of even gate lattices is a simple group. In particular, this proved\footnote{Moreover, it gives a much easier proof in the $\mathbb{Z}$-case than the one given in~\cite{HKS2022} for full shifts.}  that for all mixing $\mathbb{Z}$-SFT's, the subgroup of stabilized inert automorphisms is simple, since mixing is equivalent to the eventual filling property over $\mathbb{Z}$.

In~\cite{schmiedingLocalMathcalEntropy2022}, an entropy-like quantity called local $\mathcal{P}$ entropy was introduced for pointed groups. The primary motivation there was its use an invariant for stabilized automorphism groups of mixing shifts of finite type. It was proved in~\cite{schmiedingLocalMathcalEntropy2022} that the local $\mathcal{P}$ entropy of the stabilized automorphism group of a mixing shift of finite type recovers the topological entropy of the underlying shift system up to a rational multiple. This was used to give in~\cite{schmiedingLocalMathcalEntropy2022} a complete classification of the stabilized automorphism groups of full shifts over $\mathbb{Z}$: two such groups are isomorphic if and only if their alphabet sizes agree up to some powers.

Here we introduce a variant of local $\mathcal{P}$ entropy which we call local $\mathcal{Q}$ entropy. It takes as input a pointed group, some parameter (what we call here an entropy class selector), and (when defined) produces a nonnegative real number. While it's definition is highly influenced by local $\mathcal{P}$ entropy, there are some subtle but distinct differences, and overall it serves to generalize the theory of local $\mathcal{P}$ entropy (in the sense that one may obtain local $\mathcal{P}$ for pointed groups $(G,g)$ using local $\mathcal{Q}$ entropy). In particular, it is suitable for certain pointed groups arising from actions of groups more general than just $\mathbb{Z}$, whereas for local $\mathcal{P}$ entropy we must take groups pointed at cyclic subgroups.

Here we apply these local $\mathcal{Q}$ entropy quantities to the stabilized automorphism groups of $\mathbb{Z}^{d}$-shifts of finite type. We need a suitable mixing condition, and in our setting the condition of contractibility again is sufficient. We prove that for such systems, there are input parameters for which the local $\mathcal{Q}$ entropy recovers, up to a rational multiple, the topological entropy of the underlying system.

\begin{theorem}\label{thm:introentropyrecover}
Let $(X,\mathbb{Z}^{d})$ be a nontrivial contractible $\mathbb{Z}^{d}$-SFT. There exist entropy class selectors $\mathcal{Q}$ for which
$$h_{\mathcal{Q}}\left( \autinfty(X,\mathbb{Z}^{d}),\mathcal{Z}_{X} \right) = h_{top}(X,\mathbb{Z}^{d}).$$
\end{theorem}

The entropy class selectors that are used in the above theorem can be made flexible enough to work for several shifts of finite type at once. Using this, we can then prove the following.

\begin{theorem}\label{thm:introentropyrationalratio}
Suppose that $X$ and $Y$ are contractible $\mathbb{Z}^{d}$-shifts of finite type and that
$$\Psi \colon \autinfty(X,\mathbb{Z}^{d}) \to \autinfty(Y,\mathbb{Z}^{d})$$
is an isomorphism. Then $\frac{\htop(X)}{\htop(Y)} \in \mathbb{Q}$.
\end{theorem}

As a consequence of the above, we obtain the following classification of stabilized automorphism groups for full shifts over $\mathbb{Z}^{d}$.

\begin{theorem}\label{thm:introclassification1}
Let $\mathcal{A}$ and $\mathcal{B}$ be finite alphabets of size at least two and let $d \ge 1$. Then $\autinfty(\mathcal{A}^{\mathbb{Z}^{d}},\mathbb{Z}^{d})$ and $\autinfty(\mathcal{B}^{\mathbb{Z}^{d}},\mathbb{Z}^{d})$ are isomorphic if and only if there exists $m,n \in \mathbb{N}$ such that $|\mathcal{A}|^{m} = |\mathcal{B}|^{n}$.
\end{theorem}

A key component in the proof of Theorem~\ref{thm:introentropyrationalratio} (and hence also Theorem~\ref{thm:introclassification1}, which relies on it), is the finitary Ryan's theorem stated above, namely Theorem~\ref{thm:introFinitaryRyanZd}. In particular, it is the crucial piece to promote abstract isomorphisms to pointed isomorphisms among stabilized automorphism groups.


An outline of the paper is as follows. Section 2 gives some background and notation. Local $\mathcal{Q}$ entropy is introduced in Section 3, and several fundamental properties are proved there. Section 4 introduces the BEEPS property, introduces the suitable entropy class selectors we will use, and proves the upper bound needed for Theorem~\ref{thm:introentropyrecover}. Section 5 discusses gate lattices and proves the lower bound needed for Theorem~~\ref{thm:introentropyrecover}. Section 6 contains all of the material on the finitary Ryan's Theorems. Section 7 proves some necessary results about ghost centers, and finally all of the ingredients are brought to bear for the classification theorem, Theorem~\ref{thm:introclassification1}, in Section 8.

\section{Notation and background}
Let $G$ be a group; throughout, all groups are always assumed to be countable. A $G$-system $(X,G)$ is a compact metric space $X$ equipped with a $G$ action by homeomorphisms of $X$, i.e. a homomorphism $G \to \textrm{Homeo}(X)$. If $H \subset G$ is a subgroup, then $(X,H)$ means the system with the action of $H$ induced as a subgroup of $G$. Given $G$-systems $(X,G)$ and $(Y,G)$, by a map $\pi \colon (X,G) \to (Y,G)$ we will always mean a continuous function $\pi \colon X \to Y$ such that $\pi(g \cdot x) = g \cdot \pi(x)$ for all $g \in G$ and $x \in X$.

For a finite alphabet $\mathcal{A}$, the set $\mathcal{A}^{G}$ equipped with the product topology (using the discrete topology on $\mathcal{A}$) is a compact space. For a point $x \in \mathcal{A}^{G}$ and $a \in G$ we write $x_{a}$ to mean $x(a)$. The group $G$ acts on $\mathcal{A}^{G}$ via $g \cdot x = \sigma_{g}(x)$ where $\sigma_{g}(x)_{a} = x_{ag}$, and the system $\left(\mathcal{A}^{G},G\right)$ is called the full $\mathcal{A}$-shift over $G$. A system is a full shift (over $G$) if it is a full $\mathcal{A}$-shift over $G$ for some alphabet $\mathcal{A}$, and we occasionally use the phrase \emph{full $n$-shift over $G$} by which we mean a full shift $\mathcal{B}^{G}$ for some alphabet $\mathcal{B}$ where $|\mathcal{B}| = n$.

A \emph{pattern} is an element $w \in \mathcal{A}^{E}$ for some subset $E \subset G$. Given $x \in \mathcal{A}^{G}$ and $E \subset G$, we denote the restriction of $x$ to $E$ by $x|E$. Then any $x \in \mathcal{A}^{G}$ determines a pattern in $\mathcal{A}^{E}$ via $x|E$.

By a \emph{$G$-subshift} we mean a compact $G$-invariant subset of a full shift over $G$. A $G$-subshift $X \subset \mathcal{A}^{G}$ is a \emph{shift of finite type} if there exists a finite set $W$ of finite patterns (called the set of forbidden patterns) defined over $\mathcal{A}$ such that $X$ consists of all points in $\mathcal{A}^{G}$ whose $G$-orbits contain no pattern in $W$. Without loss of generality one may always assume a set of forbidden patterns defining a shift of finite type all have the same domain; we call such a domain a \emph{window} for the shift of finite type. We can always assume a window contains the identity element $\textrm{id}$ of $G$. When $G = \mathbb{Z}^{d}$, a number $R$ is a \emph{window size} for $X$ if there is a window $D$ containing the origin which defines $X$ such that $||v||_{\infty} \le R$ for all $v \in D$.

We will need to impose some kind of mixing conditions on our SFT's. Over $\mathbb{Z}$, the most common such condition is that of topologically mixing. Over more general groups such as $\mathbb{Z}^{d}$ however, there are a myriad of mixing conditions. Key for us here will be the following definition, introduced in~\cite{PoirierSaloContractible}.

\begin{definition}
A subshift $X \subset A^G$ is contractible if there exists a block map $h : \{0,1\}^G \times X \times X \to X$ such that $h(0^G, x, y) = x$ and $h(1^G, x, y) = y$ for all $x, y \in X$.
\end{definition}

Recall a subshift $X \subset \mathcal{A}^{G}$ is \emph{strongly irreducible} if there exists a finite set $E \subset G$ such that for all pairs of finite subsets $A,B \subset G$ such that $EA \cap B = \emptyset$ every $x,y \in X$, there exists $z \in X$ such that $z|A = x|A$ and $z|B = y|B$. Contractible subshifts are strongly irreducible, and the condition can be thought of as a notion of strong irreducibility with the requirement that the gluing is produced via a block map. Over $\mathbb{Z}$, a subshift of finite type is contractible if and only if it is topologically mixing. Over $\mathbb{Z}^{d}$, in the SFT case, contractibility is equivalent to the map extension property given in~\cite{MeyerovitchEmbeddingTheorem}. Moreover, in~\cite{PoirierSaloContractible} it is proved that a contractible subshift over a finitely generated residually finite group with finite periodic asymptotic dimension (in particular, $\mathbb{Z}^{d}$) has dense periodic points. Finally, it is apparent from the definition that if $(X,G)$ is a contractible $G$-subshift and $H \subset G$ is finite index, then $(X,H)$ is contractible as well. In particular, if $(X,G)$ is a contractible subshift then $(X,H)$ has dense periodic points for every finite index subgroup $H \subset G$.

Given a $G$-subshift $X$ and finite subset $F \subset G$, the topological entropy of $(X,G)$ is defined by
$$\htop(X,G) = \lim_{n \to \infty} \frac{1}{|F_{n}|} \log |\{x|F_{n} \mid x \in X\}|$$
where $F_{n}$ is a F\o{}lner sequence for $G$. It is a classical fact that this does not depend on the F\o{}lner sequence chosen. Here and throughout, $\log$ is taken base two. It is immediate that if $X = \mathcal{A}^{G}$ is a full shift, then $\htop(X,G) = \log|\mathcal{A}|$.

We will occasionally use of the following well-known fact.
\begin{proposition}\label{prop:indexisofact}
Let $G$ be a countable group and $L \subset G$ a finite index subgroup. If $\mathcal{A}$ is a finite alphabet and $\mathcal{B} = \mathcal{A}^{[G \colon L]}$, then $(\mathcal{A}^{G},L)$ is topologically conjugate to $(\mathcal{B}^{L},L)$. Moreover, if $X$ is any $G$-subshift, then $\htop(X,L) = [G \colon L] \cdot \htop(X,G)$.
\end{proposition}

Given a finite set $F \subset \mathbb{Z}^{d}$, $S > 0$ and $n \ge 1$, we define $$\partial_{S}F = \{x \in \mathbb{Z}^{d} \mid |x-y|_{\infty} \le S \textrm{ for some } y \in \mathbb{Z}^{d} \setminus F\}.$$

We also define the $r$-ball around $F$ by
$$\ball_{r}(F) = \{x \in \mathbb{Z}^{d} \mid |x-y|_{\infty} \le r \textrm{ for some } y \in F\}.$$

Given a subset $S$ of a group $G$ we write $C_{G}(S)$ for the centralizer of $S$ in $G$, or when the underlying group $G$ is clear from context, we simply write $C(S)$.

\subsection{Automorphisms and stabilized automorphism groups}
An automorphism of a $G$-system $(X,G)$ is a homeomorphism $\phi \colon X \to X$ such that $\phi g = g \phi$ for all $g \in G$, and we denote the group of all automorphisms of $(X,G)$ by $\aut(X,G)$.

For a group $G$, we let $\mathcal{I}(G)$ denote the collection of finite index subgroups of $G$ and $\finorm(G)$ the collection of finite index normal subgroups of $G$.

We define the stabilized automorphism group of a system $(X,G)$ by
$$\autinfty(X,G) = \bigcup_{F \in \mathcal{I}(G)}\aut(X,F).$$

We assume throughout that $G$ acts faithfully on $X$. In this case, there is an isomorphic copy of the center of $G$ in $\aut(X,G)$, which we usually denote by $\mathcal{Z}_{X}$ when the group is clear. Note that $\mathcal{Z}_{X}$ is then also a subgroup of $\autinfty(X,G)$.

\begin{proposition}\label{prop:prop1}
Let $G$ be a group and let $(X,G)$ be a $G$-system.
\begin{enumerate}
\item
If $H$ is a subgroup of $G$, then for the $H$-system $(X,H)$ obtained by restricting the $G$-action to $H$, we have $\aut(X,G) \subset \aut(X,H)$. In particular, $\autinfty(X,G) = \bigcup_{F \in \finorm(G)}\aut(X,F)$.
\item
If $H \subset G$ is a finite index subgroup, then $\autinfty(X,H) = \autinfty(X,G)$.
\end{enumerate}
\end{proposition}
\begin{proof}
That $\aut(X,G) \subset \aut(X,H)$ for a subgroup $H \subset G$ is straightforward to check. The second part then follows from the fact that every finite index subgroup of a group $G$ contains a subgroup which is finite index and normal in $G$.

For $(2)$, since any finite index subgroup of $H$ is finite index in $G$, we get that $\autinfty(X,H) \subset \autinfty(X,G)$. Now if $K \subset G$ is finite index, then $K \cap H$ is finite index in $H$ and contained in $K$, so using part (1) we have
$$\aut(X,K) \subset \aut(X,K \cap H) \subset \autinfty(X,H).$$
This implies $\autinfty(X,G) \subset \autinfty(X,H)$.

\end{proof}

Let $G_{1},G_{2}$ be groups. An isomorphism of two systems $(X,G_{1}), (Y,G_{2})$ is a pair $(\Psi,\Phi)$ where $\Psi \colon G_{1} \to G_{2}$ is an isomorphism of groups and $\varphi \colon X \to Y$ is a homeomorphism such that
$$\varphi (g x) = \Psi(g)\varphi (x), \qquad g \in G_{1}, x \in X.$$
If two systems $(X,G),(Y,G)$ are isomorphic by a pair $(\textnormal{id},\varphi)$ then such a $\Phi$ is called a topological conjugacy, and we say $(X,G)$ and $(Y,G)$ are topologically conjugate. Note that if $(\Psi,\varphi) \colon (X,G_{1}) \to (Y,G_{2})$ is an isomorphism of systems and $H_{1} \subset G_{1}$ is any subgroup, then $(\Psi|_{H_{1}}),\varphi) \colon (X,H_{1}) \to (Y,\Psi(H_{1}))$ is an isomorphism of systems.

Given groups $G_{1},G_{2}$, we say two systems $(X,G_{1}), (Y,G_{2})$ are virtually isomorphic if there exists finite index subgroups $H_{i} \subset G_{i}$ such that $(X,H_{1})$ and $(Y,H_{2})$ are isomorphic.
\begin{proposition}\label{prop:someproperties1}
Let $G_{1},G_{2}$ be groups and $(X,G_{1}), (Y,G_{2})$ be systems.
\begin{enumerate}
\item
If $(X,G_{1}), (Y,G_{2})$ are isomorphic, then $\aut(X,G_{1})$ and $\aut(Y,G_{2})$ are isomorphic groups.
\item
If $(X,G_{1})$ and $(Y,G_{2})$ are virtually isomorphic, then the stabilized groups $\autinfty(X,G_{1})$ and $\autinfty(Y,G_{2})$ are isomorphic groups.
\end{enumerate}
\end{proposition}
\begin{proof}
For $(1)$, suppose $(\Psi,\varphi) \colon (X,G_{1}) \to (Y,G_{2})$ is an isomorphism. It is straightforward to check that the map
\begin{equation*}
\begin{gathered}
\Phi_{*} \colon \aut(X,G_{1}) \to \aut(Y,G_{2})\\
\Phi_{*} \colon \alpha \mapsto \varphi \alpha \varphi^{-1}
\end{gathered}
\end{equation*}
defines an isomorphism. For part $(2)$, suppose there are finite index subgroups $H_{i} \subset G_{i}$ such that $(X,H_{1})$ and $(Y,H_{2})$ are isomorphic via $(\Psi,\varphi)$. By Proposition~\ref{prop:prop1}, we have $\autinfty(X,G_{1}) = \autinfty(X,H_{1})$ and $\autinfty(Y,G_{2}) = \autinfty(Y,H_{2})$, so it suffices to show that that $\autinfty(X,H_{1})$ and $\autinfty(Y,H_{2})$ are isomorphic.

\sloppy For every pair of finite index subgroups $K_{i} \subset H_{i}$ the induced map ${(\Psi|_{K_{1}}),\varphi) \colon (X,K_{1}) \to (Y,\Psi(K_{1}))}$ is an isomorphism, so by the previous part, $(\Psi|_{K_{1}},\varphi)$ induces an isomorphism between $(X,K_{1})$ and $(Y,\Psi(K_{1}))$. Notice that these isomorphisms are compatible in the sense that if $L_{1} \subset K_{1}$, then the restriction of the isomorphism $(\Psi|_{L_{1}},\varphi)_{*} \colon \aut(X,L_{1}) \to \aut(Y,\Psi(L_{1}))$ to $\aut(X,K_{1})$ agrees with $(\Psi|_{K_{1}},\varphi)_{*}$. Thus the collection of all $\{(\Psi_{K},\varphi)\}_{K \in \mathcal{I}(H)}$ assemble to give an injective homomorphism $\autinfty(X,H_{1}) \to \autinfty(Y,H_{2})$ Finally, since the map $K \mapsto \Psi(K)$ defines a bijection between $\mathcal{I}(H_{1})$ and $\mathcal{I}(H_{2})$, this map is in fact an isomorphism.
\end{proof}

\textbf{Example: } Fix $d \ge 1$, and let $(X_{n},\mathbb{Z}^{d})$ denote the full shift on some symbol set of size $n$ over $\mathbb{Z}^{d}$. Then $(X_{2},\mathbb{Z}^{d})$ and $(X_{4},\mathbb{Z}^{d})$ are virtually isomorphic. Indeed, if we let $H$ denote the subgroup of $\mathbb{Z}^{d}$ integrally spanned by $\{2e_{1},\ldots,e_{d}\}$ where $e_{i}$ are the standard basis vectors in $\mathbb{Z}^{d}$, then $(X_{2},H)$ and $(X_{4},\mathbb{Z}^{d})$ are isomorphic systems by Proposition~\ref{prop:indexisofact}. By part (2) of Proposition \ref{prop:someproperties1}, it follows that $\autinfty(X_{2},\mathbb{Z}^{d})$ and $\autinfty(X_{4},\mathbb{Z}^{d})$ are isomorphic groups.

However, the groups $\aut(X_{2},\mathbb{Z}^{d})$ and $\aut(X_{4},\mathbb{Z}^{d})$ are not isomorphic, which can be seen as follows. In the group $\Aut(X_4, \Z^d)$, the center is the group of shifts, isomorphic to $\Z^d$~\cite{Hochman2010}, and the shift $\sigma_{e_1}$ has a square root $g$. Suppose we had $\Aut(X_4, \Z^d) \cong \Aut(X_2, \Z^d)$. The center of $\Aut(X_2, \Z^d)$ has again as its center the group of shifts $\Z^d$, and the isomorphism implies that there is some free generating set $\{\sigma_{\vec v_1}, \ldots, \sigma_{\vec v_d}\}$ of this shift group such that $\sigma_{v_1}$ has a square root $g$.

Consider the index-$2$ subgroup $H = \langle \sigma_{2 \vec v_1}, \ldots, \sigma_{\vec v_d} \rangle \subset \Z^d$ of the shift group of $X_2$. The set of $H$-fixed points $(X_{2})_{H}$ has two elements, and $\sigma_{\vec v_1}$ acts nontrivially on it. But $g$ is a permutation of this set of size $2$, so $g^2$ must act trivially on $(X_2)_H$, a contradiction.\qed

\section{Local $\mathcal{Q}$ entropy}
In this section we introduce the definition of local $\mathcal{Q}$ entropy for pointed groups. This, together with the finitary Ryan's theorem in the next section, will be our main tool for recovering the topological entropy from the stabilized automorphism groups.

Let $L$ be a countable residually finite amenable group. By an \emph{$L$-tower} we mean a nested sequence $\mathcal{H} = (H_{m})_{m=1}^{\infty}$ of finite index normal subgroups of $L$ such that $\bigcap_{m=1}^{\infty}H_{m} = \{\textrm{id}\}$, and for which there exists a sequence of finite sets $D_{m} \subset L$ such that all of the following hold:
\begin{enumerate}
\item
For each $m$, $D_{m}$ is a fundamental domain for $H_{m}$ in $L$.
\item
$\{\textrm{id}\} \subset D_{m} \subset D_{m+1}$ for all $m \ge 1$.
\item
$L = \bigcup_{m=1}^{\infty} D_{m}$.
\item
For every $n > m$, we have $D_{n} = \bigcup_{g \in D_{n} \cap H_{m}}gD_{m}$.
\item
The sequence $D_{m}$ is a F\o{}lner sequence for $L$.
\end{enumerate}

Such a sequence $D_{m}$ is called a \emph{F\o{}lner section} for the tower $\mathcal{H}$.

The following classical result shows that every exhaustive nested sequence of finite index normal subgroups has a subsequence which forms a tower.

\begin{lemma}[\cite{WeissMonotileable}]\label{lemma:towerdomains}
Let $L$ be a countably residually finite group and $(H_{m})_{m=1}^{\infty}$ a sequence of finite index normal subgroups such that $\bigcap_{m=1}^{\infty}H_{m} = \{\textrm{id}\}$. There exists a subsequence $m_{i}$ and a sequence of finite sets $D_{i} \subset L$ such that all of the following hold:
\begin{enumerate}
\item
For each $i$, $D_{i}$ is a fundamental domain for $H_{m_{i}}$ in $L$.
\item
$\{\textrm{id}\} \subset D_{i} \subset D_{i+1}$ for all $i \ge 1$.
\item
$L = \bigcup_{i=1}^{\infty} D_{i}$.
\item
For every $j > i$, we have $D_{j} = \bigcup_{g \in D_{j} \cap H_{m_{i}}}gD_{i}$.
\end{enumerate}
In addition, if $L$ is amenable, than the sequence $D_{i}$ can be chosen to be a F\o lner sequence for $L$.
\end{lemma}

For a group $L$ we let $\mathfrak{T}(L)$ denote the collection of $L$-towers, we let $\fol(G)$ denote the set of F\o{}lner sequences $\mathcal{F} \colon \mathbb{N} \to \fins(G)$, and for an $L$-tower $\mathcal{H}$ we let $\mathfrak{F}(\mathcal{H})$ denote the collection of F\o{}lner sections for $\mathcal{H}$.


The local $\mathcal{Q}$ entropy takes as input, in addition to a pointed group, a parameter defined as follows.

\begin{definition}
An \emph{entropy class selector} (ECS) for $L$ is a function $\mathcal{Q}$ which, for each  F\o{}lner sequence $\mathcal{F} = (F_{m})$ in $L$ and $k \in \mathbb{N}$, assigns a class of finite groups $\mathcal{Q}(\mathcal{F},k)$ which is closed under isomorphism for every $k$.
\end{definition}
If $\mathcal{Q}$ is an ECS for $L$ and $\Psi \colon L \to L^{\prime}$ is an isomorphism, then we define the ECS $\Psi(\mathcal{Q})$ for $L^{\prime}$ by $\Psi(\mathcal{Q})(\mathcal{F},k) = \mathcal{Q}(\Psi^{-1}(\mathcal{F}),k)$. If $\mathcal{Q}^{\prime}$ is another ECS for $L$, we write $\mathcal{Q} \subset \mathcal{Q}^{\prime}$ if for every $(\mathcal{F},m) \in \fol(L) \times \mathbb{N}$ we have $\mathcal{Q}(\mathcal{F},m) \subset \mathcal{Q}^{\prime}(\mathcal{F},m)$.

The following definition plays a key role.

\begin{definition}
Let $(G,L)$ be a pointed group. Fix now an ECS $\mathcal{Q}$ for $L$, let $\mathcal{H} = (H_{m})$ be an $L$-tower, let $\mathcal{D} = (D_{m})$ be a F\o{}lner section for $\mathcal{H}$, and let $\mathcal{K} = \{K_{m}\}_{m=1}^{\infty}$ be a sequence of subgroups of $G$. We say $\mathcal{K}$ is \emph{$\mathcal{H}$,$\mathcal{D}$-locally $\mathcal{Q}$} if all of the following hold:
\begin{enumerate}
\item
$K_{m} \subset C_{G}(H_{m})$ for every $m \ge 1$.
\item
$K_{m} \in \mathcal{Q}(\mathcal{D},m)$ for all but finitely many $m$.
\item
There exists $\textrm{id} \ne \varphi$ such that $\varphi \in K_{m}$ for all $m \ge 1$.
\end{enumerate}
\end{definition}

In general, we call a sequence \emph{$(K_{m})_{m=1}^{\infty}$ $L$-locally $\mathcal{Q}$} if it is $\mathcal{H}$,$\mathcal{D}$-locally $\mathcal{Q}$ for some $L$-tower $\mathcal{H}$ and F\o{}lner section $\mathcal{D}$ for $\mathcal{H}$.

Given an $L$-tower $\mathcal{H}$ and F\o{}lner section $\mathcal{D}$, we let $\mathfrak{L}_{\mathcal{Q}}(G,\mathcal{H},\mathcal{D})$ denote the set of all $\mathcal{H}$,$\mathcal{D}$-locally $\mathcal{Q}$ sequences of subgroups of $G$. Occasionally, when the group $G$ is clear from context, we suppress it from the notation and just write $\mathfrak{L}_{\mathcal{Q}}(\mathcal{H},\mathcal{D})$.

\begin{remark}
Suppose $(G,L)$ is a pointed group and $\mathcal{Q}$ is an ECS for $L$ which has the property that for every $\mathcal{F} \in \fol(G)$, $\mathcal{Q}(\mathcal{F},m) \subset \mathcal{Q}(\mathcal{F},m+1)$ holds for all $m$ (as will be the case for our applications). Then $G$ contains an $L$-locally $\mathcal{Q}$ sequence with respect to some tower and F\o{}lner section if and only if there exists a finite index normal subgroup $H$ in $L$ and nontrivial finite group $K \subset C_{G}(H)$ such that $K \in \mathcal{Q}(\mathcal{F},m)$ for some $m$. For if $H$ and $K$ are as such, by Lemma~\ref{lemma:towerdomains} there exists an $L$-tower $\mathcal{H}$ and associated F\o{}lner section $\mathcal{D}$ for which $H_{1} \subset H$. Then $\{\textrm{id}\} \ne K \subset C(H_{1})$ and $K \in \mathcal{Q}(\mathcal{F},n)$ for all $n \ge m$, so the constant sequence $K_{m} = K$ is $\mathcal{H},\mathcal{D}$-locally $\mathcal{Q}$.
\end{remark}

We now define the local $\mathcal{Q}$ entropy of a pointed group $(G,L)$ with respect to an ECS for $L$.


\begin{definition}
Let $\mathcal{Q}$ be an ECS for $L$ and $(G,L)$ a pointed group which contains an $L$-locally $\mathcal{Q}$ sequence. We define the local $\mathcal{Q}$ entropy of $(G,L)$ by
$$h_{\mathcal{Q}}(G,L) = \sup_{\mathcal{H} \in \mathfrak{T}(L)} \sup_{\mathcal{D} \in \mathfrak{F}(\mathcal{H})} \sup_{(K_{m}) \in \mathfrak{L}_{\mathcal{Q}}(G,\mathcal{H},\mathcal{D})} \limsup_{m \to \infty} \frac{1}{[L \colon H_{m}]} \log \log |K_{m}|.$$
\end{definition}

Note that by definition, if $(K_{m}) \in \mathfrak{L}_{\mathcal{Q}}(G,\mathcal{H},\mathcal{D})$, then $|K_{m}| > 1$ for every $m$, so $\log \log |K_{m}|$ is defined.

When defined, the local $\mathcal{Q}$ entropy satisfies some useful invariance properties which we record now.

For pointed groups $(G_{1},L_{1})$ and $(G_{2},L_{2})$, by a pointed monomorphism (resp. isomorphism) $\Psi \colon (G_{1},L_{1}) \to (G_{2},L_{2})$ we mean an injective (resp. bijective) group homomorphism $\Psi \colon G_{1} \to G_{2}$ such that $\Psi(L_{1}) = L_{2}$.

If $\Psi \colon (G_{1},L_{1}) \to (G_{2},L_{2})$ is a pointed monomorphism, then  $\Psi(L_{1}) = L_{2}$, so $\Psi|_{L_{1}} \colon L_{1} \to L_{2}$ is an isomorphism. Then if $\mathcal{Q}$ is an ECS for $L_{1}$, we denote by $\Psi(\mathcal{Q})$ the ECS $\Psi|_{L_{1}}(\mathcal{Q})$ on $L_{2}$.
\begin{proposition}\label{prop:pentropyproperties}
Let $(G_{1},L_{1})$ and $(G_{2},L_{2})$ be pointed groups and suppose $\mathcal{Q}$ is an ECS for $L_{1}$.
\begin{enumerate}
\item
Suppose $\Psi \colon (G_{1},L_{1}) \to (G_{2},L_{2})$ is a pointed monomorphism. If the local $\mathcal{Q}$ entropy for $(G_{1},L_{1})$ exists, then the local $\Psi(\mathcal{Q})$ entropy for $(G_{2},L_{2})$ exists, and
$$h_{\mathcal{Q}}(G_{1},L_{1}) \le h_{\Psi(\mathcal{Q})}(G_{2},L_{2}).$$
\item
If $\Psi \colon (G_{1},L_{1}) \to (G_{2},L_{2})$ is a pointed isomorphism and the local $\mathcal{Q}$ entropy for $(G_{1},L_{1})$ exists, then the local $\Psi(\mathcal{Q})$ entropy for $(G_{2},L_{2})$ exists, and
$$h_{\mathcal{Q}}(G_{1},L_{1}) = h_{\Psi(\mathcal{Q})}(G_{2},L_{2}).$$
\item
If $\mathcal{Q}^{\prime}$ is an ECS for $L_{1}$ such that $\mathcal{Q} \subset \mathcal{Q}^{\prime}$, then
$$h_{\mathcal{Q}}(G_{1},L_{1}) \le h_{\mathcal{Q}^{\prime}}(G_{1},L_{1}).$$
\end{enumerate}
\end{proposition}
\begin{proof}
Let $\mathcal{H} = (H_{m})$ be a an $L_{1}$-tower, $\mathcal{D} = (D_{m})$ a F\o{}lner section for $\mathcal{H}$, and $(K_{m})$ an $\mathcal{H}$,$\mathcal{D}$-locally $\mathcal{Q}$ sequence in $G_{1}$. Then $\mathcal{H}^{\prime} = (H_{m}^{\prime}) = (\Psi(H_{m}))$ is an $L_{2}$-tower, and $\mathcal{D}^{\prime} = \Psi(\mathcal{D})$ is a F\o{}lner section for $\mathcal{H}^{\prime}$. By definition, we have $\Psi(\mathcal{Q})(\Psi(\mathcal{D}),m) = \mathcal{Q}(\mathcal{D},m)$. Since $\Psi$ is an isomorphism, and $\mathcal{Q}(\mathcal{D},m)$ is closed under isomorphism for every $m$, we get that $\Psi(K_{m}) \in \mathcal{Q}(\mathcal{D},m)$ for all $m$ sufficiently large, and hence $\Psi(K_{m}) \in \Psi(\mathcal{Q})(\Psi(\mathcal{D}),m)$ for such $m$ as well. It follows that the sequence $\Psi(K_{m})$ is $\mathcal{H}^{\prime},\mathcal{D}^{\prime}$-locally $\Psi(\mathcal{Q})$ in $G_{2}$. Since $\Psi$ is injective, we have
$$|\Psi(K_{m})| = |K_{m}|$$
for every $m$, from which it follows that $h_{\mathcal{Q}}(G_{1},L_{1}) \le h_{\Psi(\mathcal{Q})}(G_{2},L_{2})$.

For part (2), note that by part (1) we have $h_{\mathcal{Q}}(G_{1},L_{1}) \le h_{\Psi(\mathcal{Q})}(G_{2},L_{2})$. But since $\Psi^{-1}(\Psi(\mathcal{Q})) = \mathcal{Q}$, we also obtain $h_{\Psi(\mathcal{Q})}(G_{2},L_{2}) \le h_{\mathcal{Q}}(G_{1},L_{1})$ again using part (1).

For part (3), again suppose $\mathcal{H} = (H_{m})$ is an $L_{1}$-tower, $\mathcal{D} = (D_{m})$ a F\o{}lner section for $\mathcal{H}$, and $(K_{m})$ an $\mathcal{H}$,$\mathcal{D}$-locally $\mathcal{Q}$ sequence in $G_{1}$. Then $K_{m} \in \mathcal{Q}(\mathcal{D},m)$ and $\mathcal{Q} \subset \mathcal{Q}^{\prime}$ implies $K_{m} \in \mathcal{Q}^{\prime}(\mathcal{D},m)$, and it follows that $(K_{m})$ is also $\mathcal{H},\mathcal{D}$-locally $\mathcal{Q}^{\prime}$ as well, from which (3) follows.
\end{proof}

\begin{remark}
We briefly describe the relationship between the definition of local $\mathcal{Q}$ entropy and local $\mathcal{P}$ entropy as defined in~\cite{schmiedingLocalMathcalEntropy2022}. There, for the notion of a pointed group $(G,L)$ one always takes $L$ to be an infinite cyclic subgroup of $G$, so one considers pointed groups of the form $(G,\langle g \rangle)$ where $g \in G$. Given a class $\mathcal{P}$ of finite groups closed under isomorphism, we may consider the constant ECS defined by $\mathcal{Q}_{\mathcal{P}}(\mathcal{F},m) = \mathcal{P}$ for every F\o{}lner sequence for $\mathbb{Z}$. Then using the terminology of~\cite{schmiedingLocalMathcalEntropy2022}, if $H$ is a $g$-locally $\mathcal{P}$ subgroup in $G$, then the sequence $H_{m} = H \cap C_{G}(g^{m})$ is an $\mathcal{H},\mathcal{D}$-locally $\mathcal{Q}_{\mathcal{P}}$ sequence as defined here. Thus we may always recover the local $\mathcal{P}$ entropy $h_{\mathcal{P}}(G,g)$ of a pointed group $(G,\langle g \rangle)$ using local $\mathcal{Q}$ entropy with constant ECS's.
\end{remark}

Our applications here will be in the context of $(\autinfty(X),\mathcal{Z}_{X})$ where $X$ is a shift of finite type. For these groups, we will use specific types of ECS's, which we define now.

For an amenable group $G$, we say a function $f \colon \fol(G) \times \mathbb{N} \to \mathbb{R}$ is \emph{controlled} if $\lim_{m \to \infty}\frac{\log f(\mathcal{F},m)}{|F_{m}|} = 0$ for every F\o{}lner sequence $\mathcal{F} = (F_{m})$ in $G$. We observe that if $f,g \colon \fol(G) \times \mathbb{N} \to \mathbb{R}$ are controlled functions then $\max\{f,g\}$ is also controlled.

\begin{definition}
Given $M,N \ge 1$, define $\mathcal{P}^{s}(M,N)$ to be the class of finite groups which can be written as a product $\prod_{i=1}^{r}A_{i}$ such that $r \le N$, $A_{i}$ is a nontrivial simple group for every $i$, and for every $i,j$ we have
$$\frac{\log |A_{i}|}{\log |A_{j}|} \le M.$$

Given  functions $f,g \colon \fol(G) \times \mathbb{N} \to \mathbb{N}$, we define the ECS $\mathcal{Q}(f,g)$ on $\fol(G) \times \mathbb{N}$ by $\mathcal{Q}(f,g)(\mathcal{F},m) = \mathcal{P}^{s}(f(\mathcal{F},m),g(\mathcal{F},m))$.
\end{definition}

We note that if $f_{1},f_{2},g_{1},g_{2} \colon \fol(G) \times \mathbb{N} \to \mathbb{N}$ have the property that for every $\mathcal{F} \in \fol(G)$ and $m \in \mathbb{N}$, $f_{1}(\mathcal{F},m) \le f_{2}(\mathcal{F},m)$ and $g_{1}(\mathcal{F},m) \le g_{2}(\mathcal{F},m)$, then $\mathcal{Q}(f_{1},g_{1}) \subset \mathcal{Q}(f_{2},g_{2})$.

It will be convenient to have notation for the following functions: given a group $G$, real number $R > 0$, and integer $\kappa > 0$, define
\begin{equation*}
\begin{gathered}
v_{R,\kappa} \colon \fins(G) \to \mathbb{N}\\
v_{R,\kappa}(F) = \kappa^{|\partial_{R}F|}.
\end{gathered}
\end{equation*}
The functions $v_{R,\kappa}$ naturally define controlled functions $v_{R,\kappa} \colon \fol(G) \times \mathbb{N} \to \mathbb{N}$ by $v_{R,\kappa}(\mathcal{F},n) = v_{R,\kappa}(F_{n})$, since
$$\lim_{n \to \infty} \frac{1}{|F_{n}|}\log v_{R,\kappa}(F_{n}) = \lim_{n \to \infty} \frac{1}{|F_{n}|} \log \kappa^{|\partial_{R}F_{n}|} = \lim_{n \to \infty} \frac{1}{|F_{n}|} |\partial_{R}F_{n}| \log \kappa = 0.$$
Oftentimes we will abuse notation and write simply $v_{R,\kappa}(\mathcal{F},n) = v_{R,\kappa}(F_{n})$.

\begin{remark}
Let $(G,L)$ be a pointed amenable group and consider a pair of functions $f,g \colon \fol(G) \times \mathbb{N} \to \mathbb{N}$ and the associated ECS $\mathcal{Q}(f,g)$ as defined above. Then the local $\mathcal{Q}(f,g)$ entropy of $(G,L)$ exists if and only if there exists a finite index normal subgroup $H$ in $L$ and nontrivial simple group $K \subset C(H)$. Indeed, given such $H$ and $K$, by Lemma~\ref{lemma:towerdomains} there exists an $L$-tower $\mathcal{H}$ and associated F\o{}lner section $\mathcal{D}$ for which $H_{1} \subset H$. Then $K \subset C(H_{1})$ and the constant sequence $K_{m} = K$ is $\mathcal{H},\mathcal{D}$-locally $\mathcal{Q}(f,g)$.
\end{remark}



\section{Upper Bound and the BEEPS property}
We now introduce some the relevant mixing properties we will use, and the entropy class selectors that we will build using them. We then make our way toward proving the upper bound for Theorem~\ref{thm:introentropyrecover}. Throughout this section we work with subshifts over the group $\mathbb{Z}^{d}$. For $u \in X|\partial_{S}F$, define
$$E(u) = \{w \in X|F \mid w|\partial_{S}F = u\}.$$
In other words, $|E(u)|$ is the number of ways to extend $u$ to an $X$-admissible pattern on $F$.

\subsection{D-mixing and the BEEPS property}

\begin{definition}
Let $(X,T)$ be a $\mathbb{Z}^{d}$-shift of finite type.
We say that $(X,\mathbb{Z}^{d})$ has the \emph{boundary extension entropy production similarity property} (BEEPS) if for every $S > 0$ there exists a controlled function $f_{S} \colon \fol(\mathbb{Z}^{d}) \times \mathbb{N} \to \mathbb{R}$ such that for every F\o{}lner sequence $\mathcal{F} = (F_{n})$ in $\mathbb{Z}^{d}$, the following is satisfied:
\[
2 \frac{|E(u)|}{|E(v)|}\frac{\log|E(u)|}{\log|E(v)|} \le f_{S}(\mathcal{F},n)
\textrm{ for all } n \textrm{ sufficiently large and every } u,v \in X|\partial_{S}(F_{n}).
\]
We call such an $f_{S}$ an \emph{$S$-BEEPS function} for $(X,\mathbb{Z}^{d})$, and call $f \colon \fol(\mathbb{Z}^{d}) \times \mathbb{N} \to \mathbb{R}$ a BEEPS function for $(X,\mathbb{Z}^{d})$ if it is an $S$-BEEPS function for some $S$. If $(X,L)$ has the BEEPS property for every finite index subgroup $L \subset \mathbb{Z}^{d}$, then we say $(X,\mathbb{Z}^{d})$ has \emph{strong BEEPS}.
\end{definition}

\sloppy We note several things. First, by definition, a BEEPS function for $(X,\mathbb{Z}^{d})$ is always controlled. Second, if $f$ is an $S$-BEEPS function for $(X,\mathbb{Z}^{d})$ and ${g \colon \fol(\mathbb{Z}^{d}) \times \mathbb{N} \to \mathbb{R}}$ is a controlled function such that for every F\o{}lner sequence $\mathcal{F}$ and $n \in \mathbb{N}$ we have $f(\mathcal{F},n) \le g(\mathcal{F},n)$, then $g$ is also an $S$-BEEPS function for $(X,\mathbb{Z}^{d})$. Lastly, for our purposes here, the 2 that appears in the definition of a BEEPS function could be replaced with any constant $C>1$; since 2 will suffice for us, for simplicity we use that.

\sloppy If $A \subset G$ then $B \supset A$ is a \emph{mixing set} for $A$ if any globally valid contents of $A$ and $G \setminus B$ are compatible, meaning
\[ \forall y, y' \in X: \exists z \in X: z|A = y|A \wedge z|G \setminus B = y'|G \setminus B. \]
In~\cite{BGM2020}, a subshift $X$ is defined to be \emph{D-mixing} (short for Dobrushin-mixing) if for some F\o{}lner sequence $F_n$ we have a mixing set $\bar F_n$ for $F_n$ satisfying ${|\bar F_n \setminus F_n| = o(|F_{n}|)}$. It is not clear that this property is independent of the choice of the F\o{}lner set (though it holds for \emph{all} F\o{}lner sequences in many cases), so we need to parametrize it by the F\o{}lner set. The $F_n$ here are not the sequence we use for this, as for BEEPS we want to mix into things inside rather than outside, so we massage it a little.

\begin{lemma}\label{lemma:Dmixingfolners}
A subshift $X$ is D-mixing if and only if it admits a F\o{}lner sequence $C_n$ such that there is a sequence of subsets $B_n \subset C_n$ such that $B_{n}$ is a F\o{}lner sequence, $|C_n \setminus B_n| = o(|C_n|)$, and for all $S$ and all large enough $n$, $C_n \setminus \partial_S(C_n)$ is a mixing set for $B_n$.
\end{lemma}

\begin{proof}
Suppose $X$ is D-mixing and let $F_n$ be the F\o{}lner sequence in the definition. Let $\bar F_n$ be the mixing sets and let $k_n$ be a slowly growing sequence of natural numbers tending to infinity. Then $C_n = \ball_{k_n}(\bar F_n)$ is also automatically a F\o{}lner sequence, and we can choose $B_n = F_n$ (any set containing a mixing set for $B_n$ is mixing, and $C_n \setminus \partial_S(C_n) \supset \bar F_n$ for large $n$ since eventually $k_n \geq S$). Choosing $k_{n}$ sufficiently slowly growing then we have $|C_{n} \setminus B_{n}| = o(|C_{n}|)$ as well.

For the other direction, pick the $F_n = B_n$ (which is automatically a F\o{}lner sequence) and $\bar F_n = C_n$.
\end{proof}

Thus it makes sense to say that $X$ is \emph{D-mixing for the F\o{}lner sequence $C_n$} if such $B_n$ as in the lemma exist. In the following, by an SFT being nontrivial we mean it has at least two points.

\begin{proposition}\label{prop:dmixingforallbeeps}
Let $(X,T)$ be a nontrivial SFT on an amenable group. Suppose that $(X,T)$ is D-mixing for every F\o{}lner sequence. Then $(X,T)$ has the BEEPS property.
\end{proposition}



\begin{proof}
First note that a nontrivial SFT with D-mixing has positive entropy: we certainly have some mixing set $B$ for the singleton set $\{1\}$ containing the identity of $G$, so for any F\o{}lner set of $G$ we can pick a maximal packing of $B$'s to produce entropy.

Now suppose that $\mathcal{C} = C_{n}$ is a F\o{}lner sequence, let $S$ be arbitrary, and consider any $u \in X|\partial_{S}(C_n)$. It is well-known that for an SFT with positive entropy $h$, for any F\o{}lner sequence $F_n$, the number of patterns on $F_n$ is $2^{h|F_n| + g(F_{n})}$ where $g(F_{n}) = o(|F_{n}|)$. Thus
\[ |E(u)| \leq \vert X|C_n \vert = 2^{h|C_n| + g(C_{n})} \]
and
$$\log |E(u)| \le h|C_{n}| + g(C_{n}).$$

Choose some $B_n \subset C_n$ given by the definition of D-mixing. For large enough $n$, $C_{n} \setminus \partial_{S}(C_n)$ is mixing for $B_n$, giving
\[ |E(u)| \geq \vert X|B_n \vert = 2^{h|B_n| + g(B_{n})} \]
and
$$\log |E(u)| \ge h|B_{n}| + g(B_{n}).$$

It follows that
\[ \frac{|E(u)|}{|E(v)|} \le 2^{h|C_n| + g(C_{n}) - h|B_n| - g(B_{n})} = 2^{h|C_n \setminus B_n| + g(C_{n})-g(B_{n})}\]
\[ \le 2^{h|C_n \setminus B_n| + g(C_{n}) + |g(B_{n})|} \]
and hence
$$2 \frac{|E(u)|}{|E(v)|} \frac{\log |E(u)|}{\log |E(v)|} \le 2^{h|C_n \setminus B_n| + g(C_{n}) + |g(B_{n})|+1}\left(\frac{h|C_{n}| + g(C_{n})}{h|B_{n}| + g(B_{n})}\right).$$
Thus defining
\[ f(\mathcal{C},n) = 2^{h|C_n \setminus B_n| + g(C_{n}) + |g(B_{n})|+1}\left(\frac{h|C_{n}| + g(C_{n})}{h|B_{n}| + g(B_{n})}\right) \]
we have
$\log f(\mathcal{C},n) = \big( h|C_{n} \setminus B_{n}| + g(C_{n}) + |g(B_{n})|+1 \big) \log \left(\frac{h|C_{n}| + g(C_{n})}{h|B_{n}| + g(B_{n})}\right)$.

Since $|C_n \setminus B_n| = o(|C_n|)$, it follows that $\frac{\log(f(\mathcal{C},n))}{|C_{n}|} \overset{n \rightarrow \infty}\longrightarrow 0$.
\end{proof}

\begin{proposition}\label{prop:howtogetthebeeps}
Suppose $(X,T)$ is a $\mathbb{Z}^{d}$-SFT which is strongly irreducible. Then $(X,T)$ has the BEEPS property. In particular, this holds if $(X,T)$ is contractible.
\end{proposition}
\begin{proof}
Let $C_{n}$ be a F\o{}lner sequence. To see D-mixing for $C_{n}$, let $k_n$ be a slowly growing sequence of natural numbers tending to infinity and set $B_{n} = \{g \in C_{n} \mid d(g,\partial C_{n}) \ge k_{n}\}$. Then $B_{n}$ is a F\o{}lner sequence and $|C_n \setminus B_n| = o(|C_n|)$, and for all $S$ and all large enough $n$ we have that $k_{n} \ge S$, so using strong irreducibility $C_n \setminus \partial_S(C_n)$ is a mixing set for $B_n$ as needed.

The last part holds since a contractible SFT is strongly irreducible.
\end{proof}

A key class of ECS's for our applications are the following. Suppose $X$ is a $\mathbb{Z}^{d}$-SFT with the BEEPS property, let $S > 0$ be real, $\kappa > 0$ an integer, and $f_{S}$ a BEEPS functions for $X$. Then $\mathcal{Q}(f_{S},v_{S,\kappa})$ is an ECS for $\left(\autinfty(X),\mathcal{Z}_{X}\right)$.
\subsection{Upper bound}
We are now prepared to prove the upper bound.

Let $G$ be a countable group and $\alpha \colon G \curvearrowright X$ an action of $G$ on the compact metric space $X$. Given a finite index subgroup $F \subset G$, we let $p_{F}(\alpha) =  |\fix(F)|$. Suppose $\alpha$ has the property that $p_{F}(\alpha)$ is finite for every finite index subgroup $F \subset G$. Given a $G$-tower $\mathcal{H} = (H_{m})$, we define the exponential growth rate of periodic points along $\mathcal{H}$ as the quantity
$$\rho_{\mathcal{H}}(\alpha) = \limsup_{m \to \infty} \frac{1}{[G \colon H_{m}]} \log \max\{1,p_{H_{m}}(\alpha)\}.$$

Recall a $G$-system $(X,G)$ is expansive if there exists $\epsilon > 0$ such that for any pair $x \ne y$ in $X$, there exists $g \in G$ such that $d(g \cdot x, g \cdot y) \ge \epsilon$. If $(X,G)$ is an expansive $G$-system with expansivity constant $\epsilon$ and $F$ is a finite index subgroup of $G$, then $|\fix(F)| < \infty$. Indeed, if $C = \{c_{i}\}_{i=1}^{[G \colon F]}$ are a set of left coset representatives for $F$, then there exists $\delta > 0$ such that for any $x,y \in X$, if $d(x,y) < \delta$ then $d(c_{i} \cdot x, c_{i} \cdot y) < \epsilon$ for all $1 \le i \le [G \colon F]$. Since $(X,G)$ is expansive, it follows that if $x \ne y$ are in $\fix(F)$, then $d(x,y) > \delta$, so $\fix(F)$ must be finite since $X$ is compact.


We say an SFT $(X,\mathbb{Z}^{d})$ has \emph{strong density of periodic points} if for every nested sequence of finite index subgroups $F_{n} \subset \mathbb{Z}^{d}$ with $\bigcap_{n=1}^{\infty}F_{n} = \{\textrm{id}\}$, the set $\bigcup_{n=1}^{\infty}\fix(F_{n})$ is dense in $X$.


\begin{theorem}\label{thm:upperboundentropy}
Suppose $(X,\mathbb{Z}^{d})$ is a shift of finite type which has the BEEPS property, and has strong density of periodic points. Let $\kappa$ be a posive integer, $R,S > 0$ be real, $f_{S}$ a BEEPS function for $(X,\mathbb{Z}^{d})$ and $\mathcal{Q}(f_{S},v_{R,\kappa})$ the associated ECS for $\mathbb{Z}^{d}$. Then
$$h_{\mathcal{Q}(f_{S},v_{R,\kappa})}(\autinfty(X,\mathbb{Z}^{d}),\mathcal{Z}_{X}) \le \htop(X,\mathbb{Z}^{d}).$$
\end{theorem}
We note that the theorem applies in particular to a large class of $\mathbb{Z}^{d}$ SFT's (outlined below), including full shifts.


Before beginning the proof of the theorem, we record a useful lemma.
\begin{lemma}\label{lemma:rhoentropybound}
Let $G$ be a countable residually finite amenable group acting expansively on $X$ and let $\mathcal{H} = (H_{m})$ be a $G$-tower. Then $\rho_{\mathcal{H}}(\alpha) \le \htop(X,G)$.
\end{lemma}
\begin{proof}
Since $\mathcal{H}$ is a $G$-tower, we may choose a F\o{}lner section $\mathcal{D} = (D_{m})$ for $\mathcal{H}$ in $G$. Let $\epsilon > 0$ be an expansivity constant for the action of $G$ on $X$. By expansivity, for every $i$, if $x \ne y \in \fix(H_{m})$ then there exists $g \in D_{m}$ such that $d(g \cdot x, g \cdot y) > \epsilon/2$. This means $\fix(H_{m})$ is an $\epsilon/2$-separated set of cardinality $|\fix(H_{m})|$.
\end{proof}

\begin{proof}[Proof of Theorem~\ref{thm:upperboundentropy}]
Throughout the proof we will write simply $\mathcal{Z}$ for $\mathcal{Z}_{X}$.
Let $\mathcal{H} = (H_{m})$ be a $\mathcal{Z}$-tower, $\mathcal{D} = (D_{m})$ a F\o{}lner section for $\mathcal{H}$, and $(K)_{m}$ an $\mathcal{H}$,$\mathcal{D}$-locally $\mathcal{Q}(f_{S},v_{S})$ subgroup of $\autinfty(X,G)$. Our goal is to prove that
$$\limsup_{m \to \infty} \frac{1}{[\mathcal{Z} \colon H_{m}]} \log \log |K_{m}| \le \htop(X,\mathbb{Z}^{d}).$$

For any finite index subgroup $F$ of $\mathcal{Z}$, we may consider the $F$-periodic point homomorphism
\begin{equation}\label{eqn:pprep}
\begin{gathered}
\pi_{F} \colon \aut(X,F) \to \sym(\fix(F))\\
\pi_{F} \colon \alpha \mapsto \alpha|_{\fix(F)}.
\end{gathered}
\end{equation}
Since $(K_{m})$ is $\mathcal{H}$,$\mathcal{D}$-locally $\mathcal{Q}(f_{S},v_{R,\kappa})$, we may choose $M_{1}$ such that for all $m \ge M_{1}$, the subgroup $K_{m} \subset C(H_{m})$ and $K_{m} \in \mathcal{P}^{s}\left(f_{S}(\mathcal{D},m),v_{R,\kappa}(D_{m})\right)$. Note that for every $m \ge 1$, we have that $K_{m} \subset \aut(X,H_{m})$ since $K_{m} \subset C(H_{m})$. Fix an $m \ge M_{1}$. Recall for $F \in \mathcal{I}(\mathcal{Z})$ we defined $p_{F} = |\fix(F)|$.

We pause now for a key lemma.
\begin{lemma}\label{lemma:normalrigidity}
If $|K_{m}| > \left(p_{H_{m}}!\right)^{f_{S}(\mathcal{D},m)v_{R,\kappa}(D_{m})}$ for some $m \ge M_{1}$, then ${K_{m} \subset \ker \pi_{H_{m}}}$.
\end{lemma}
\begin{proof}
Since $K_{m} \in \mathcal{P}^{s}(f_{S}(\mathcal{D},m),v_{R,\kappa}(D_{m}))$ we may write $K_{m} = \prod_{i=1}^{r}A_{i}$ where each $A_{i}$ is a finite simple group, $\frac{\log |A_{i}|}{\log |A_{j}|} \le f_{S}(\mathcal{D},m)$ for every $i,j$, and $r \le v_{R,\kappa}(D_{m})$. Then $\log |K_{m}| = \sum_{i=1}^{r}\log |A_{i}|$. The hypothesis $|K_{m}| > \left(p_{H_{m}}!\right)^{f_{S}(\mathcal{D},m)v_{R,\kappa}(D_{m})}$ implies
$$\sum_{i=1}^{r}\log |A_{i}| > f_{S}(\mathcal{D},m)v_{R,\kappa}(D_{m}) \log \left(p_{H_{m}}! \right)$$
so
$$v_{R,\kappa}(D_{m}) \max_{i}\{\log |A_{i}|\} > r \max_{i}\{\log |A_{i}|\} > f_{S}(\mathcal{D},m)v_{R,\kappa}(D_{m}) \log \left(p_{H_{m}}!\right).$$
But we must have, for every $j$,
$$f_{S}(\mathcal{D},m) \log |A_{j}| \ge \max_{i} \{\log |A_{i}|\}$$
so
$$f_{S}(\mathcal{D},m)v_{R,\kappa}(D_{m}) \log |A_{j}| \ge v_{R,\kappa}(D_{m}) \max_{i}\{\log |A_{i}|\} > f_{S}(\mathcal{D},m)v_{R,\kappa}(D_{m}) \log \left(p_{H_{m}}!\right).$$
Thus
$$\log |A_{j}| > \log \left( p_{H_{m}}! \right)$$
and hence
$$|A_{j}| > \left( p_{H_{m}}! \right).$$
Fix $1 \le j \le r$, and recall $\pi_{H_{m}} \colon \aut(X,H_{m}) \to \sym(\textrm{Fix}(H_{m}))$. Since the target of this map has size $p_{H_{m}}!$, it follows that the subgroup $\tilde{A}_{j} \subset K_{m}$ which consists of $A_{j}$ in the $j$th coordinate and the identity in all other coordinates must have nontrivial intersection with $\ker \pi_{H_{m}}$. But normality of $\ker \pi_{H_{m}}$ then implies $\tilde{A}_{j} \subset \ker \pi_{H_{m}}$. Since this happens for every $j$, we get that $K_{m} \subset \ker \pi_{H_{m}}$ as desired.
\end{proof}

Continuing with the proof of Theorem~\ref{thm:upperboundentropy}, fix $\epsilon > 0$, and note that $\htop(X,\mathbb{Z}^{d})$ is finite, and hence $\rho_{\mathcal{H}}(X,\mathbb{Z}^{d})$ is finite as well.
We will prove the following claim.\\

\emph{Claim: } There exists $M_{2} \ge M_{1} \ge 1$ such that $|K_{m}| \le \left(p_{H_{m}}!\right)^{f_{S}(\mathcal{D},m)v_{R,\kappa}(D_{m})}$ for all $m \ge M_{2}$.




First we prove the claim. Suppose instead the claim does not hold. Then there exists a sequence $m_{k} \to \infty$ such that for all $k$ we have
$$|K_{m_{k}}| > \left(p_{H_{m_{k}}}!\right)^{f_{S}(\mathcal{D},m)v_{R,\kappa}(D_{m})}.$$
Then by Lemma~\ref{lemma:normalrigidity}, this implies $K_{m_{k}} \subset \ker \pi_{H_{m_{k}}}$. Since $\mathcal{K} = (K_{m})$ is $\mathcal{H}$,$\mathcal{D}$-locally $\mathcal{Q}(f_{S},v_{R,\kappa})$, by (3) of the definition there exists some $\textrm{id} \ne \varphi \in K_{m}$ for all $m \ge 1$. Thus for every $k$ we have must have $\varphi \in \ker \pi_{H_{m_{k}}}$ and hence $\varphi$ acts trivially on $\textrm{Fix}(H_{m_{k}})$. By definition, the base action of $\mathbb{Z}^{d}$ on $X$ agrees with the action of $\mathcal{Z}$, and thus by the assumption that $\bigcup_{k=1}^{\infty}\textrm{Fix}(H_{m_{k}})$ is dense in $X$ in the statement of the theorem, this means $\varphi$ acts trivially on a dense subset of $X$. Since $\varphi$ is continuous, this contradicts that $\varphi \ne \textrm{id}$, proving the claim.

Now we proceed to finish the proof of the theorem using the claim. Suppose $m \ge M_{2}$. Then
$$|K_{m}| \le \left(p_{H_{m}}!\right)^{f_{S}(\mathcal{D},m)v_{R,\kappa}(D_{m})} \le p_{H_{m}}^{p_{H_{m}}f_{S}(\mathcal{D},m)v_{R,\kappa}(D_{m})}$$
and hence
$$
\frac{1}{[\mathcal{Z} \colon H_{m}]} \log \log |K_{m}| \le \frac{1}{[\mathcal{Z} \colon H_{m}]} \log \left( p_{H_{m}}f_{S}(\mathcal{D},m)v_{R,\kappa}(D_{m}) \log p_{H_{m}}\right)$$
\begin{equation}\label{eqn:upperboundproof1}
= \frac{1}{[\mathcal{Z} \colon H_{m}]} \Big( \log p_{H_{m}} + \log f_{S}(\mathcal{D},m) + \log v_{R,\kappa}(D_{m}) + \log \log p_{H_{m}} \Big).
\end{equation}
Since $\rho(X,\mathbb{Z}^{d})$ is finite and $\mathcal{Z}$ can be identified with $\mathbb{Z}^{d}$, it follows that
$$\limsup_{m \to \infty} \frac{1}{[\mathcal{Z} \colon H_{m}]} \log \log p_{H_{m}}=0.$$

Now since $f_{S}$ is a a BEEPS function and $D_{m}$ is a fundamental domain for $H_{m}$, we know that
$$\lim_{m \to \infty} \frac{1}{[\mathcal{Z} \colon H_{m}]}\log f_{S}(\mathcal{D},m) = \lim_{m \to \infty} \frac{1}{|D_{m}|}\log f_{S}(\mathcal{D},m) = 0.$$
Moreover, since $D_{m}$ is a F\o{}lner sequence we have
$$\lim_{m \to \infty} \frac{1}{[\mathcal{Z} \colon H_{m}]}\log v_{R,\kappa}(D_{m}) = \lim_{m \to \infty} \frac{1}{|D_{m}|}\log v_{R,\kappa}(D_{m})$$
$$= \lim_{m \to \infty} \frac{1}{|D_{m}|}\log \kappa^{|\partial_{R}D_{m}|} = \lim_{m \to \infty} \frac{|\partial_{R}D_{m}|}{|D_{m}|} \log \kappa = 0.$$

From this we get
$$\limsup_{m \to \infty} \frac{1}{[\mathcal{Z} \colon H_{m}]} \log \log |K_{m}| \le \limsup_{m \to \infty} \frac{1}{[\mathcal{Z} \colon H_{m}]} \log p_{H_{m}}.$$

By Lemma~\ref{lemma:rhoentropybound}, the right hand side is bounded above by $\htop(X,\mathbb{Z}^{d})$, completing the proof.

\end{proof}

\begin{proposition}\label{prop:contractiblegivesstuff}
Contractible $\mathbb{Z}^{d}$-SFT's have strong density of periodic points and the BEEPS property.
\end{proposition}
\begin{proof}
That a contractible SFT has the BEEPS property is Proposition~\ref{prop:howtogetthebeeps}. The strong density of periodic points follows from the fact that the induced action from a finite index subgroup is still contractible, together with~\cite[Theorem 6.10]{PoirierSaloContractible}.
\end{proof}

This immediately implies the following.

\begin{theorem}
Contractible $\mathbb{Z}^{d}$-SFT's satisfy Theorem~\ref{thm:upperboundentropy}.
\end{theorem}



\section{Gate lattices and achieving entropy}
The last section shows that the local $\mathcal{Q}(f_{S},v_{R,\kappa})$ entropy of the pointed group $\left(\autinfty(X,\mathbb{Z}^{d}),\mathcal{Z}_{X}\right)$ is bounded above by the topological entropy of the system for any $\mathbb{Z}^{d}$-SFT $(X,\mathbb{Z}^{d})$ with the BEEPS property and strongly dense periodic points and any BEEPS function $f_{S}$. The goal of this section is to obtain the lower bound for the local $\mathcal{Q}$ entropy, and hence prove that for contractible shifts of finite type over $\mathbb{Z}^{d}$, the local $\mathcal{Q}$ entropy recovers exactly the topological entropy of the SFT. Now we must be slightly more specific with our choices of parameters $R,S,\kappa$.

\begin{theorem}\label{thm:fullshiftentropycalc}
Let $(X,\mathbb{Z}^{d})$ be an SFT on the alphabet $\mathcal{A}$ with strong density of periodic points and the BEEPS property. Let $R \ge S$ be greater than the window size for $X$, $f_{S}$ be a BEEPS function for $(X,\mathbb{Z}^{d})$, and $\kappa \ge |\mathcal{A}|$. Then
$$h_{\mathcal{Q}(f_{S},v_{R,\kappa})}\left( \autinfty(X,\mathbb{Z}^{d}),\mathcal{Z}_{X} \right) = h_{top}(X,\mathbb{Z}^{d}).$$
\end{theorem}

In particular, by Proposition~\ref{prop:contractiblegivesstuff} the above holds for contractible $\mathbb{Z}^{d}$-SFT's.
\begin{theorem}\label{thm:entropycalccontractibles}
Let $(X,\mathbb{Z}^{d})$ be a nontrivial contractible $\mathbb{Z}^{d}$-SFT on alphabet $\mathcal{A}$. Let $R \ge S$ be greater than the window size for $X$ and $f_{S}$ be a BEEPS function for $(X,\mathbb{Z}^{d})$. Then for $\kappa \ge |\mathcal{A}|$ we have
$$h_{\mathcal{Q}(f_{S},v_{R,\kappa})}\left( \autinfty(X,\mathbb{Z}^{d}),\mathcal{Z}_{X} \right) = h_{top}(X,\mathbb{Z}^{d}).$$
\end{theorem}

In light of Theorem~\ref{thm:upperboundentropy}, to prove Theorem \ref{thm:fullshiftentropycalc} it suffices to show that the local $\mathcal{Q}(f_{S},v_{R,\kappa})$ entropy $h_{\mathcal{Q}(f_{S},v_{S,\kappa})}\left( \autinfty(X,\mathbb{Z}^{d}),\mathcal{Z}_{X} \right)$ is bounded below by $\htop(X,\mathbb{Z}^{d})$. For this, it suffices to find a $\mathcal{Z}_{X}$-tower $\mathcal{H}$, a F\o{}lner section $\mathcal{D}$ for $\mathcal{H}$, and an $\mathcal{H},\mathcal{D}$-locally $\mathcal{Q}(f_{S},v_{R,\kappa})$ sequence $K_{m}$ in $\autinfty(X,\mathcal{Z}_{X})$ which is $\mathcal{Z}_{X}$-locally $\mathcal{Q}(f_{S},v_{R,\kappa})$ such that
$$\limsup_{m \to \infty} \frac{1}{[\mathcal{Z}_{X} \colon H_{m}]} \log \log |K_{m}| \ge \htop(X,\mathbb{Z}^{d}).$$
We will do this using \emph{gate lattices}, introduced in~\cite{SaloGateLattices}, which serve as a generalization of the subgroups of stabilized simple automorphisms introduced in~\cite{schmiedingLocalMathcalEntropy2022}.

\subsection{Gates and gate lattices}
For now, fix a countable residually finite amenable group $G$ and let $X \subset \mathcal{A}^{G}$ be a $G$-subshift. A \emph{gate} on $(X,G)$ is a homeomorphism $\chi \colon X \to X$ for which there exists a finite set $N \subset G$ and a permutation $\tau_{\chi} \colon X|_{N} \to X_{N}$  such that, for all $x \in X$, we have $\chi(x)_{g} = \tau(x_{g})$ for $g \in N$ and $\chi(x)_{g} = x_{g}$ for all $g \in G \setminus N$. Such an $N$ for $\chi$ is called a \emph{strong neighborhood} for $\chi$. For $g \in G$ and a gate $\chi$ we use the notation $\chi^{g} = \sigma_{g^{-1}} \chi \sigma_{g}$. Note that if $N$ is a strong neighborhood for $\chi$, then $Ng$ is a strong neighborhood for $\chi^{g}$, and $\chi^{g}$ acts by the same permutation on $Ng$ as $\chi$ does on $N$. Moreover, if $N_{1}, N_{2}$ are strong neighborhoods for gates $\chi_{1},\chi_{2}$ and $N_{1} \cap N_{2} = \emptyset$, then $\chi_{1} \chi_{2} = \chi_{2} \chi_{1}$. In particular, if for some gate $\chi$ with strong neighborhood $N$ and some $g \in G$ we have $gN \cap N = \emptyset$, then $\chi^{g} \chi = \chi \chi^{g}$.

Suppose $H$ is a finite index subgroup of $G$ and let $D$ be fundamental domain for $H$ in $G$. Suppose further that $\chi$ is a gate with strong neighborhood $D$. Then for each distinct $h_{1}, h_{2} \in H$, we have $h_{1}D \cap h_{2}D = \emptyset$, and hence $\chi^{h_{1}}\chi^{h_{2}} = \chi^{h_{1}}\chi^{h_{2}}$. Define
$$\chi^{H} = \lim_{F \subset H, |F|< \infty} \prod_{h \in F} \chi^{h}$$
where the limit means pointwise convergence and the infinite product means the infinite composition. Lemmas 12 and 13 in~\cite{SaloGateLattices} show that this function is well-defined, continuous, and belongs to $\aut(X,H)$. Following~\cite{SaloGateLattices}, a homeomorphism of the form $\chi^{H}$ for some finite index subgroup $H \subset G$ is called a \emph{gate lattice}.

We introduce some notation. Recall for a finite $D \subset G$, $S>0$ and pattern $u \in X|\partial_{S}D$, $E(u)$ denotes the set of $X$-admissible extensions of $u$ to a $D$-pattern in $X$.
\begin{definition}
Let $X$ be a $G$-SFT. Given a finite set $D \subset G$ we let $\gate(X,D)$ denote the group of gates with strong neighborhood $D$. Denote by $\gate^{e}(X,D)$ the subgroup of $\gate(X,D)$ which correspond to gates whose permutations of $\mathcal{A}^{D}$ are even. If $H$ is a finite index subgroup $H \subset G$ for which $D$ is a fundamental domain, then we define
$$\mathfrak{G}(X,H,D) = \{\chi^{H} \mid \chi \in \gate(X,D)\}$$
and the subgroup
$$\mathfrak{G}^{\textrm{e}}(X,H,D) = \{\chi^{H} \mid \chi \in \gate^{\textrm{e}}(X,D)\}.$$

Now suppose $S > 0$ is larger than a window size for $X$ and let $u \in X|\partial_{S}D$. If $\chi \in \gate(X,D)$ satisfies $\chi|_{\partial_{S}D} = \textrm{id}$, then $\chi|_{D}$ defines an associated permutation $\tau_{\chi} \in \sym(E(u))$, and we define
$$\gate^{e}(X,u) = \{\chi \in \gate(X,D) \mid \chi|_{\partial_{S}D} = \textrm{id} \textrm{ and } \tau_{\chi} \in \Alt(E(u))\}.$$
Finally, we define
$$\mathfrak{G}^{e}(X,u) = \{\chi^{H} \mid \chi \in \gate^{e}(X,u)\}.$$
\end{definition}
Note that if $H$ is a finite index subgroup of $G$ with $D$ a fundamental domain and $u \in X|\partial_{S}D$, then $\mathfrak{G}^{e}(X,u)$ is a subgroup of $\mathfrak{G}^{e}(X,H,D)$. Moreover, by definition, $\gate^{e}(X,u)$ is isomorphic to $\Alt(E(u))$. Thus $\mathfrak{G}^{e}(X,u)$ is also isomorphic to $\Alt(E(u))$, since it is straightforward to check that $\mathfrak{G}^{e}(X,u)$ is isomorphic to $\gate^{e}(X,u)$.

The proof of the following is straightforward, and we leave it to the reader.

\begin{proposition}\label{prop:gatelatticeproperties1}
Let $X$ be a $\mathbb{Z}^{d}$-SFT. If $H \subset G$ is finite index and $D$ is a fundamental domain for $H$, then $\mathfrak{G}(X,H,D)$ is a subgroup of $\aut(X,H)$. Moreover, if $S > 0$ is greater than the window size for $X$ and $u \in X|\partial_{S}D$, then $\mathfrak{G}^{e}(X,u)$ is a subgroup of $\aut(X,H)$ which is isomorphic to $\Alt(E(u))$. Finally, if $u_{1} \ne u_{2}$ are both in $X|\partial_{S}D$, then the subgroups $\mathfrak{G}^{e}(X,u_{1})$ and $\mathfrak{G}^{e}(X,u_{2})$ strongly commute, in the sense that for every pair $\phi_{i} \in \mathfrak{G}^{e}(X,u_{i})$ we have $\phi_{1} \phi_{2} = \phi_{2} \phi_{1}$, and hence for any set $\{u_{1},\ldots,u_{k}\} \subset X|\partial_{S}D$, the product $\prod_{i=1}^{k}\mathfrak{G}^{e}(X,U_{i})$ sits naturally as a subgroup of $\aut(X,H)$.
\end{proposition}
The D-mixing condition is sufficient for us to generate many nontrivial $\mathfrak{G}^{e}(X,u)$.
\begin{lemma}\label{lemma:fillingsmakesimple}
Suppose $X$ is a $\mathbb{Z}^{d}$-SFT, let $\mathcal{H} = (H_{n})$ be a $\mathbb{Z}^{d}$ tower and $\mathcal{C} = (C_{n})$ a F\o{}lner section for $\mathcal{H}$. Let $S > 0$ be greater than the window size for $X$. Suppose that $X$ is D-mixing for $C_{n}$. Then for every sufficiently large $n$ and every $u \in X|\partial_{S}C_{n}$, the group $\mathfrak{G}^{e}(X,u)$ is nontrivial and simple.
\end{lemma}
\begin{proof}
Lemma~\ref{lemma:Dmixingfolners} implies for large enough $n$ and every $u \in X|\partial_{S}C_{n}$ that $|E(u)| \ge 5$, so this follows from Proposition~\ref{prop:gatelatticeproperties1} and simplicity of $A_{n}$ for $n \ge 5$.
\end{proof}

We now prove the lower bound. For the statement, recall by Proposition~\ref{prop:dmixingforallbeeps} that if $X$ is $\mathbb{Z}^{d}$-SFT which is D-mixing for every F\o{}lner sequence, then $X$ has the BEEPS property.

\begin{theorem}\label{thm:lowerbound11}
Suppose that $X$ is a $\mathbb{Z}^{d}$-SFT defined on an alphabet $\mathcal{A}$ which is D-mixing for every F\o{}lner sequence and let $R \ge S > 0$ be greater than the window size for $X$. Suppose also that $f_{S}$ is an $S$-BEEPS function for $X$ and that $\kappa \ge |\mathcal{A}|$ is an integer. Then $h_{\mathcal{Q}(f_{S},v_{R,\kappa})}\left( \autinfty(X,\mathbb{Z}^{d}),\mathcal{Z}_{X} \right) \ge h_{top}(X,\mathbb{Z}^{d})$.
\end{theorem}

Before the proof, we note that the BEEPS assumption in the theorem is only needed because we wish to compute the local $\mathcal{Q}$ entropy with respect to $\mathcal{Q}(f_{S},v_{R,\kappa})$ where $f_{S}$ is a BEEPS function.

Given $n \in \mathbb{N}$ we define the finite index subgroup $H_{n} = 3^{n}\mathbb{Z}^{d}$ in $\mathcal{Z}_{X} = \mathbb{Z}^{d}$, and set $C_{n} = \{v \in \mathbb{Z}^{d} \mid |v|_{\infty} \le \frac{3^{n}-1}{2}\}$. Since $X$ is D-mixing for $C_{n}$, up to reindexing if necessary, we will assume without loss of generality that there exists $v \in X|\partial_{S}C_{1}$ such that $|E(v)| \ge 3$, and hence there exists a nontrivial even permutation $\Alt(E(v))$. Then $\mathcal{C} = (C_{n})$ is a F\o{}lner section for the tower $\mathcal{H} = (H_{n})$. Now we define
$$K_{n} = \prod_{u \in X|\partial_{S}C_{n}}\mathfrak{G}^{e}(X,u).$$

\begin{lemma}
$K_{n}$ is an $\mathcal{H},\mathcal{C}$-locally $\mathcal{Q}(f_{S},v_{S,\kappa})$ sequence.
\end{lemma}
\begin{proof}
We will verify each of the properties $(1)-(3)$ in the definition of $\mathcal{H},\mathcal{C}$-locally $\mathcal{Q}(f_{S},v_{R,\kappa})$. Property $(1)$ holds by Proposition~\ref{prop:gatelatticeproperties1}, since each $K_{n}$ belongs to $\aut(X,H_{n}) = C(H_{n})$.

For property (3), recall that we have $\textrm{id} \ne \tau \in \Alt(E(v))$, so we can define $\varphi \in \mathfrak{G}^{e}(X,v)$ using $\tau$, and then take $\tilde{\varphi} \in K_{1}$ which is given by $\varphi$ in the $v$ coordinate and $\textrm{id}$ in the other coordinates. We claim that $\tilde{\varphi} \in K_{n}$ for all $n \ge 1$. A key fact is that $\tilde{\varphi}|_{X|\partial_{S}C_{1}}$ acts by the identity. Now by induction, suppose $\tilde{\varphi} \in K_{n}$, $\tilde{\varphi}$ is induced by a gate on $C_{n}$ which is defined by an even permutation $\tau_{n}$, and $\tilde{\varphi}|_{X|\partial_{S}C_{n}}$ acts by the identity. By construction, we know that $C_{n+1} = \bigcup_{g \in C_{n+1} \cap H_{n}}gC_{n}$, and $\partial_{S}C_{n+1} \subset \bigcup_{g \in C_{n+1} \cap H_{n}}g \partial_{S}C_{n}$. Then $\tilde{\varphi}$ is induced by a gate supported on $C_{n+1}$ whose permutation $\tau_{n+1}$ is even, since the permutation $\tau_{n}$ inducing $\tilde{\phi}$ on $C_{n}$ was even, and $\tau_{n+1}$ is obtained by translated copies of $\tau_{n}$ acting on translated copies of $C_{n}$. Finally, since $\partial_{S}C_{n+1} \subset \bigcup_{g \in C_{n+1} \cap H_{n}}g \partial_{S}C_{n}$ and $\tilde{\varphi}$ acts by the identity on $\partial_{S}C_{n}$ and $\tilde{\varphi}$ commutes with each $g \in H_{n}$, it follows that $\tilde{\varphi}|_{X|\partial_{S}{C_{n+1}}}$ acts by the identity. Altogether we have that $\tilde{\varphi} \in K_{n+1}$, verifying property $(3)$ for $K_{n}$ to be $\mathcal{H},\mathcal{C}$-locally $\mathcal{Q}(f_{S},v_{R,\kappa})$.

It remains to verify property $(2)$, i.e. that for $n$ sufficiently large, we have $K_{n} \in \mathcal{Q}(f_{S},v_{R,\kappa})(\mathcal{C},n) = \mathcal{P}^{s}(f_{S}(\mathcal{C},n),v_{R,\kappa}(\mathcal{C},n))$. In other words, we must show, for $n$ sufficiently large, that $K_{n}$ is a product of at most $v_{R,\kappa}(\mathcal{C},n)$ many simple groups $A_{i}$ whose orders satisfy $\frac{\log |A_{i}|}{\log |A_{j}|} \le f_{S}(\mathcal{C},n)$. First note the definition of $K_{n}$ involves a product indexed over the set $\{u \in X|\partial_{S}C_{n}\}$, and since $R \ge S$ and $\kappa \ge |\mathcal{A}|$ we have
$$|\{u \in X|\partial_{S}C_{n}\}| \le |\mathcal{A}^{|\partial_{S}C_{n}|}| \le \kappa^{|\partial_{R}C_{n}|} = v_{R,\kappa}(\mathcal{C},n).$$
So it remains to compare the respective orders. By definition, for each $n$ and $u \in X|\partial_{S}C_{n}$, the subgroup $\mathfrak{G}^{e}(X,u)$ has order $\frac{1}{2}|E(u)|!$. By Stirling's Formula, for every pair $u,v \in X|\partial_{S}C_{n}$, we have
$$\frac{\log \left(\frac{1}{2}|E(u)|!\right)}{\log \left(\frac{1}{2}|E(v)|!\right)} = \frac{\log \frac{1}{2} + |E(u)| \log |E(u)| - |E(u)| + O(\log(|E(u)|))}{\log \frac{1}{2} + |E(v)| \log |E(v)| - |E(v)| + O(\log(|E(v)|))}.$$
Ignoring the $\log \frac{1}{2}$ terms as they are inconsequential for large enough $n$, this last term is
$$= \frac{|E(u)|}{|E(v)|} \left(\frac{\log |E(u)| - 1 + \frac{1}{|E(u)|}O(\log(|E(u)|))}{\log |E(v)| - 1 + \frac{1}{|E(v)|}O(\log(|E(v)|))} \right).$$
By assumption, $X$ is D-mixing for the F\o{}lner sequence $C_{n}$, so there exists subsets $B_{n} \subset C_{n}$ given by Lemma~\ref{lemma:Dmixingfolners} such that for large enough $n$, $C_{n} \setminus \partial_{S}(C_n)$ is mixing for $B_n$, and hence for every $u \in X|\partial_{S}C_{n}$ we have
$$|E(u)| \geq \vert X|B_n \vert = 2^{h|B_n| + g(B_{n})}.$$
Thus for large enough $n$, the terms $\frac{1}{|E(u)|}O(\log(|E(u)|))$ can be made arbitrarily small, and
$$\frac{|E(u)|}{|E(v)|} \left(\frac{\log |E(u)| - 1 + \frac{1}{|E(u)|}O(\log(|E(u)|))}{\log |E(v)| - 1 + \frac{1}{|E(v)|}O(\log(|E(v)|))} \right) \le 2\frac{|E(u)|}{|E(v)|} \frac{\log |E(u)|}{\log |E(v)|}$$
holds for all $n$ sufficiently large. Since $f_{S}$ is a BEEPS function, this last term is bounded above by $f_{S}(\mathcal{C},n)$ for all $n$ sufficiently large and every $u,v \in X|\partial_{S}C_{n}$.
\end{proof}

We now complete the proof of Theorem~\ref{thm:lowerbound11}. For $n$ sufficiently large we know that $|K_{n}| \ge \frac{1}{2}|E(u)|!$ for some $u_{n} \in X|\partial_{S}C_{n}$. Then by definition, we have
$$h_{\mathcal{Q}(f_{S},v_{R,\kappa})}(\autinfty(X),\mathcal{Z}_{X}) \ge \limsup_{n \to \infty} \frac{1}{|C_{n}|} \log \log |K_{n}|$$
$$\ge \limsup_{n \to \infty} \frac{1}{|C_{n}|} \log \log \frac{1}{2}|E(u_{n})!| = \limsup_{n \to \infty} \frac{1}{|C_{n}|} \log \log |E(u_{n})|!.$$
Now
$$\log \log |E(u_{n})|! = \log \Big(|E(u_{n})| \log |E(u_{n})| - |E(u_{n})| + O(\log(|E(u_{n})|)) \Big)$$
$$= \log |E(u_{n})| + \log \Big( \log|E(u_{n})| - 1 + \frac{1}{|E(u_{n})|} O(\log(|E(u_{n})|)) \Big).$$
Since $|E(u_{n})| \le 2^{h |C_{n}| + g(C_{n})}$ where $g(C_{n})$ is $o(|C_{n}|)$, it follows that
$$\limsup_{n \to \infty} \frac{1}{|C_{n}|} \log \Big( \log |E(u_{n})| - 1 \frac{1}{|E(u_{n})|} O(\log(|E(u_{n})|)) \Big) = 0$$
and hence
$$\limsup_{n \to \infty} \frac{1}{|C_{n}|} \log \log |(E(u_{n})|! = \limsup_{n \to \infty} \frac{1}{|C_{n}|} \log |E(u_{n})|.$$
But as shown in Lemma~\ref{lemma:Dmixingfolners}, we know there exists a F\o{}lner sequence $B_{n} \subset C_{n}$ such that $|C_{n} \setminus B_{n}| = o(|C_{n})$ and for all $n$ sufficiently large, $|E(u_{n})| \ge 2^{\htop(X) |B_{n}| + g(B_{n})}$ where $g(B_{n}) = o(|B_{n}|)$. Thus
$$\limsup_{n \to \infty} \frac{1}{|C_{n}|} \log |E(u_{n})| \ge \limsup_{n \to \infty} \frac{1}{|C_{n}|} \log 2^{\htop(X) |B_{n}| + g(B_{n})}$$
$$= \limsup_{n \to \infty}\frac{1}{|C_{n}|} \left( \htop(X)|B_{n}| + g(B_{n}) \right)$$
$$= \limsup_{n \to \infty}\frac{1}{|C_{n}|} \htop(X)|B_{n}| = \limsup_{n \to \infty} \htop(X) \Big( \frac{|C_{n}| - |C_{n} \setminus B_{n}|}{|C_{n}|} \Big) = \htop(X).$$

The proof of Theorem~\ref{thm:fullshiftentropycalc} now follows from Theorem~\ref{thm:upperboundentropy} and Theorem~\ref{thm:lowerbound11}.

\section{Finitary Ryan's theorem}
In this section we prove various finitary Ryan's theorems. These will be used later in the classification results in Section 8. We first define what we mean by the finitary Ryan's property.
\begin{definition}
Let $(X, G)$ be a subshift. We say it has a \emph{weak (resp.\ strong) finitary Ryan's theorem} if there is a finitely-generated subgroup $H \leq \Aut(X)$ such that $C_{\Aut(X)}(H) = G \cap \Aut(X)$ (resp.\ $C_{\Homeo(X)}(H) = G$). Weak and strong Ryan's theorems are defined similarly, but with $H = \Aut(X)$ in place of a f.g.\ subgroup.
\end{definition}

Ryan proved in~\cite{Ryan1,Ryan2} that every nontrivial mixing shift of finite type over $\mathbb{Z}$ has a weak Ryan's theorem. Kopra showed in~\cite{Kopra2020}, following the second author in~\cite{SaloFinitaryRyans2019} for specific full shifts, that every infinite irreducible shift of finite type over $\mathbb{Z}$ has a finitary Ryan's theorem. For more general groups, the only result known is by Hochman~\cite{Hochman2010}, who proved that topologically transitive shifts of finite type over $\mathbb{Z}^{d}$ have a Ryan's theorem.

Note that we have $Z(G) \leq G \cap \Aut(X)$ and often equality holds.

The basis of our proofs for finitary Ryan's theorems is to show that for nice enough subshifts, the action of the automorphism group is highly transitive on homoclinic points (points asymptotic to a particular fixed point).

\begin{definition}
A \emph{pointed subshift} is $(X, G, x_0)$ where $(X, G)$ is a subshift and $x_{0} \in X$ is a fixed point for the $G$-action called the \emph{zero point}. Often we don't state $x_0$ explicitly, and assume $x_0 = 0^G$, and that the symbol $0$ is in the alphabet of $X$. A point in a pointed subshift is \emph{homoclinic} if it is almost everywhere equal to the zero point. The \emph{(nonzero) support} of a point $x$ in a pointed subshift is the set of $g \in G$ such that $x_g \neq 0$.
\end{definition}

Note that this should not be confused with the standard notion of \emph{support} of a bijection on a set, which means the set of elements that it moves. (Both notions of support appear prominently in this section.)

Our main technical results establish the following property for f.g.\ subgroups of automorphism groups of certain types of subshifts:

\begin{definition}
Let $(X,G)$ be a subshift and $k \in \mathbb{N} \cup \{\infty\}$. We say a subgroup $H$ of $\textrm{Homeo}(X)$ acts \emph{$k$-orbit-transitively} on a set $Y \subset X$ if it preserves $Y$, and for any $\ell < k+1$, for any $\ell$-tuple of points $(y_1, ..., y_\ell)$ from $Y$ of which no two are equivalent up to shift, and for
any other such tuple $(z_1, ..., z_\ell)$ with the same shift stabilizers (meaning $\textrm{stab}_{G}(z_i) = \textrm{stab}_{G}(y_i)$ for all $i$), there exists $g \in H$ such that $g y_i = z_i$ for all $i$.
\end{definition}

(We note the choice of indexing $l < k+1$ is to ensure the case $k=\infty$ makes sense.)


A point $x \in X$ in a subshift is \emph{aperiodic} if its stabilizer (with respect to the shift action) is trivial. Equivalently the shift group acts freely on the shift orbit of $x$.

We will specifically prove the orbit-transitivity property for $Y$ the set of aperiodic homoclinic points, in the cases when the group is $\Z^d$ and $X$ has strong dynamical assumptions, or the group is more general and $X$ is a full shift with robust enough alphabet. It is easy to show that in a torsion-free group (such as $\Z^d$), a homoclinic point cannot have nontrivial stabilizer, so for $\Z^d$ we can equivalently take $Y$ to be the set of all nontrivial homoclinic points.


\subsection{Finitary Ryan's theorems from transitivity properties}

Since our proofs are based on homoclinic points, the cleanest results are obtained when we concentrate on homeomorphisms that preserve the homoclinic points. Write $\Homeo_0(X)$ for those self-homeomorphisms of $X$, and then $\Aut_0(X) = \Homeo_0(X) \cap \Aut(X)$. The latter is a finite-index subgroup of $\Aut(X)$, as it is simply the subgroup of automorphisms that preserve the zero point.

The following proposition establishes a weak finitary Ryan's theorem from $2$-orbit transitivity:

\begin{proposition}
\label{prop:WeakRyans}
Suppose $(X, G)$ is a pointed subshift. 
If $H\leq\Aut(X)$ acts $2$-orbit-transitively on aperiodic homoclinic points, and aperiodic homoclinic points are dense in $X$, then $C_{\Aut_0(X)}(H) = G \cap \Aut(X)$. If such f.g.\ $H$ exists and weak Ryan's theorem holds for $X$, then also weak finitary Ryan's theorem holds for $X$.
\end{proposition}

\begin{proof}
Suppose $H$ acts $2$-orbit transitively. Let $f \in \Aut_0(X)$, and let $x$ be any aperiodic homoclinic point. Suppose $f(x) = y$, so that $y$ is another aperiodic homoclinic point (since automorphisms preserve the stabilizer and $f$ preserves homoclinicity). We claim that $y$ is in the shift-orbit of $x$. Otherwise, using $2$-orbit transitivity, in $H$ we can find an element $g$ that has $x$ in its support, but not $y$. Specifically we can take $g(x, y) = (x, sy)$ where $s \in G$ is arbitrary ($y \neq sy$ since $y$ is aperiodic). But if $f$ commutes with $g$, $f$ must preserve the support of $g$, a contradiction.

Now suppose that $f(x) = sx$ for $s \in G$. Suppose that $y$ is another aperiodic homoclinic point, and let $f(y) = ty$ for some $t \in G$ (such $t$ exists by the previous paragraph). Use transitivity on aperiodic homoclinic points to find $g \in H$ such that $g(y) = x$. Then $fg(y) = sx$ and $gf(y) = g(ty) = t(g(y)) = tx$, so $s = t$ since $x$ has trivial stabilizer. We conclude that $f$ acts as the shift by $s$ on aperiodic homoclinic points. These points are dense, so $f$ is a shift map.

This proves $C_{\Aut_0(X)}(H) \leq G \cap \Aut_0(X)$, and the other inclusion is trivial, noting that $G \cap \Aut_0(X) = G \cap \Aut(X)$.

Next we use that $\Aut_0(X)$ is of finite index in $\Aut(X)$ (since it is the subgroup of elements that fix a particular fixed point), so we can find $g_1, \ldots g_n$ such that $\Aut(X) = \langle g_1, \ldots, g_n, \Aut_0(X) \rangle$. If weak Ryan's theorem holds, there clearly exists a finitely-generated subgroup $K \leq \Aut(X)$ such that $C_K(\Aut(X))$ does not contain any of the $g_i$. Then
\[ \Aut(X) \cap G \leq C_{\langle H, K \rangle}(\Aut(X)) = C_{\langle H, K \rangle}(\Aut_0(X)) \leq C_H(\Aut_0(X)) = \Aut(X) \cap G \]
where $H$ is as above. The first equality is because $g_i \not\in C_K(\Aut(X))$ for all $i$.
\end{proof}

We will now explain how to prove strong finitary Ryan's theorems. For this we will use the following ``morally weaker'' but possibly orthogonal property to $\infty$-orbit transitivity. It is in practice always easy to obtain from our proofs of high orbit transitivity:

\begin{definition}
Let $X$ be a pointed subshift. We say a subgroup $G$ of its homeomorphism group has \emph{property P} if it acts transitively on aperiodic homoclinic points, and for every aperiodic homoclinic $x$ and every $y \in X$ which is not in the shift-orbit of $x$, there exists $g \in G$ such that exactly one of $x, y$ is in the support of $g$. 
\end{definition}



Property P gives strong finitary Ryan's theorem by an exactly analogous proof as we gave for the weak one above. We state this as a lemma, as we prefer to state the main theorem below with slightly weaker assumptions.

\begin{lemma}
\label{lem:Derp}
Suppose $(X, G)$ is a pointed subshift with dense aperiodic homoclinics. Suppose that a subgroup $H$ of $\Aut(X)$ has property P. Then $C_{\Homeo(X)}(H) = G$.
\end{lemma}

\begin{proof}
The proof is the same as Proposition~\ref{prop:WeakRyans}. Let $f$ be any homeomorphism that commutes with all maps in $g \in H$. Let $x$ be an aperiodic homoclinic point. Suppose we have $f(x) = y$, and $y$ is not in the shift orbit of $x$. By property P, there exists $g \in H$ such that exactly one of $x, y$ is in the support of $x$. But if $f$ and $g$ commute, then $f$ preserves the support of $g$, a contradiction. We conclude that every homoclinic point is mapped into its shift orbit.

Then apply the second paragraph of the proof of Proposition~\ref{prop:WeakRyans} verbatim.
\end{proof}

Next we need a purely group-theoretic/combinatorial lemma.

A set $S \subset G$ is \emph{left syndetic} if $FS = G$ for some finite set $F$. Otherwise it is \emph{left thick}.

We prove a minor strengthening of Neumann's lemma (which easily follows from Neumann's lemma).

\begin{lemma}
Suppose $G$ is a group and $S_i$ are finitely many subgroups such that $S = \bigcup_i g_1S_1$ is left thick. Then one of the groups $S_i$ is of finite index.
\end{lemma}

\begin{proof}
We consider the action of $G$ on the sets $g_iS_i$ by right translation. Left thickness means that we can find a limit point $(T_1, \ldots, T_n)$ whose
union is $G$, with the Cantor topology on subsets of $G$. In other words, for a fixed $i$, we have $T_i = \lim_j g_i S_i h_j$ for some elements $h_j$,
and these sets satisfy $\bigcup_i T_i = G$.

If some $T_i$ is empty, we can drop it from the union. For other $i$, the set $g_i^{-1} T_i = \lim_j S_i h_j$ is nonempty (noting that left multiplication is continuous), say $k_i \in g_i^{-1} T_i$. Because the sets $S_i h_j$ are closed under multiplication by elements of $S_i$ on the left, the same is true for $g_i^{-1} T_i$, so $S_ik_i \subset g_i^{-1} T_i$.

In fact, $g_i^{-1}T_i$ must be equal to this coset $S_ik_i$: if $g_i^{-1} T_i$ contains $s, s'$ then some $S_i h_j$ already contains them, and $S_i h_j$ is a single right coset. We conclude that $g_i^{-1} T_i = S_i k_i$.

It follows that $G = \bigcup_i g_i S_i k_i$, and thus $G = \bigcup_i g_i k_i S_i^{k_i}$, meaning we have expressed $G$ as a union of left cosets of subgroups. Neumann's lemma implies that one of the $S_i^{k_i}$ is of finite index, thus so is $S_i$.
\end{proof}

\begin{lemma}
Let $G$ be a finitely-generated group acting on a set $X$, and let $Y \subset X$ be a finite set disjoint from another finite set $Z \subset X$. Then the set of elements $g \in G$ such that $Z \cap gY = \emptyset$ is left-syndetic.
\end{lemma}

\begin{proof}\label{lem:Syndetic}
For all $y \in Y$ that have finite orbit under $G$, we can pass to a finite index subgroup of $G$ to assume they do not move at all. Thus we may assume all $y \in Y$ have infinite orbit, equivalently their stabilizers $S_y$ are not of finite index. If $gY \cap Z \neq \emptyset$ then $gy = z$ for some $y \in Y, z \in Z$. Then the set of $g$ such that $gy = z$ is a left coset of $S_y$. We conclude that the set of $g$ such that $gY \cap Z \neq \emptyset$ is a union of finitely many left cosets of the groups $S_y$. Since these groups are not of finite index, their union cannot be left thick by the previous lemma.
\end{proof}

\begin{lemma}
\label{lem:Aperiodic}
Let $G$ be an infinite f.g.\ group. Suppose $(X, G)$ is a pointed subshift of finite type with dense homoclinic points, where $G$ acts faithfully. Then aperiodic homoclinic points are also dense there.
\end{lemma}

\begin{proof}
Suppose $U$ is open. Pick any nonzero homoclinic point $x$ in $U$. Since the shift action is proper, the stabilizer of this point is a finite subgroup $H$ of $G$.

Enumerate $H$ as $h_1, \ldots, h_k$. For each $h_i \in H$ pick a homoclinic point $y_i$ which is not stabilized by $h_i$ (using density of homoclinic points and faithfulness of the action), so that the support of $y_i$ is larger than that of $x$, and for $i < j$ the support of $y_j$ is larger than that of $y_i$. Let $R$ be the diameter of the nonzero support of $y_k$.

We claim that the set of $g \in G$ such that $gy_i$ is not stabilized is left syndetic in $G$. For this, observe that because $y_i$ also has finite support, its stabilizer $K_i = \{k_1, \ldots, k_m\}$ is finite, and $h_i \notin K_i$. As $G$ acts on $y$, the stabilizer of $gy$ evolves by conjugation of the $k_i$. Lemma~\ref{lem:Syndetic} then implies that the set of shifts $g$ such that $h_i$ does not stabilize $gy$ is left syndetic.

Now consider the point $z = x + \sum_i g_i y_i$, where $g_i$ are group elements at distance more than $100R$ from the identity and from each other. Note that for large $R$ these are in $X$. If $gz = z$, g must permute the $R$-components, which are the supports of the $x$ and $g_i y_i$. In particular since these are of different sizes, we must have $gx = x$ and thus $g = h_i \in H$.

By the same argument, we must have $g g_i y_i = h_i g_i y_i = g_i y_I$, so $h_i$ is in the stabilizer of $g_i y$. But here, the $g_i$ can be chosen arbitrarily as long as they stay far from each other. But the set of shifts $g$ such that $h_i$ does not stabilize $g y$ is left-syndetic, so we can pick the translates $g_i$ so that $h_i$ does not stabilize $g_iy_i$ for any $i$.

If the $g_i$ are also large enough, clearly we have $z \in U$.
\end{proof}

The following is immediate from combining the lemmas.

\begin{theorem}
\label{thm:StrongRyan}
Suppose $G$ is infinite f.g.\ and $(X, G)$ is a pointed subshift of finite type with dense homoclinic points, where $G$ acts faithfully. Suppose also that a subgroup $H$ of $\Aut(X)$ has property P. Then the centralizer of $H$ in $\Homeo(X)$ is $G$, in particular $X$ satisfies strong finitary Ryan's theorem.
\end{theorem}

Note that the somewhat technical lemmas above were indeed used only to conclude from dense homoclinic points and a faithful action of $G$ that aperiodic homoclinic points are dense. Of course one can simply make this natural assumption in the first place.

\subsection{Finitary Ryan's theorem for full shifts on groups}
\label{sec:FullShifts}
In this section we give a proof of strong finitary Ryan's theorem for (some) full shifts. First we recall the definition of the subgroup $\PAut(\mathcal{A}^{G}) \subset \aut(\mathcal{A}^{G})$.

Suppose $\mathcal{A} = \mathcal{A}_1 \times \mathcal{A}_2 \times \cdots \times \mathcal{A}_k$ where each $|\mathcal{A}_i|$ is prime. We notate points in $\mathcal{A}_1 \times \mathcal{A}_2 \times \cdots \times \mathcal{A}_k$ by $x_{g,i}$, $g \in G, 1 \le i \le k$. For each $g \in G$ and $i \in [1, k]$ we have the partial shift $s_{g, i}$ defined by $s_{g, i}(x)_{h, j} = x_{h, j}$ if $j \neq i$, and $s_{g, i}(x)_{h, i} = x_{hg, i}$. In addition, we have the collection $SP(\mathcal{A},G)$ of automorphisms obtained by symbol permutation of the alphabet $\mathcal{A}$. We then define $\PAut(\mathcal{A}^G)$ to be the subgroup of $\aut(\mathcal{A}^{G})$ generated by all partial shifts $\{s_{g,i} \mid g \in G,1 \le i \le k\}$ and $SP(\mathcal{A}^{G})$. This subgroup only depends on the decomposition of $|\mathcal{A}|$ and not on the choice of a bijection $\mathcal{A} \leftrightarrow \mathcal{A}_1 \times \mathcal{A}_2 \times \cdots \times \mathcal{A}_k$, as was proved in~\cite[Lemma 2.1]{SaloUniversal2022}.

\begin{definition}
We say a number $n$ is \emph{potent}\footnote{The term `powerful' is taken and relatively standard.} if it can be written as $m^k$ where $m \geq 5$ and $k \geq 4$. A potent full shift is one whose alphabet has potent cardinality.
\end{definition}

\begin{theorem}\label{thm:potentsstrongryansomgsomany}
\label{thm:FullShiftStrongRyan}
Let $G$ be infinite f.g.\ and let $(X, G)$ be a potent full shift. Then there is a f.g.\ subgroup $H$ of the automorphism group such that the centralizer of $H$ in $\Homeo(X)$ is $G$.
\end{theorem}

Before giving the proof, we explain with the corollary below why, for stabilized purposes, this is enough.

\begin{lemma}
For any infinite residually finite group $G$, and any full shift $A^G$ with $|A| \geq 2$, there is a finite index subgroup $H$ such that $(A^G, H)$ is potent.
\end{lemma}

\begin{proof}
Namely, start by picking a finite index subgroup $K$ to get $(A^G, K) \cong ((A^{[G:K]})^K, K) = (B^K, K)$ with $B$ large. Then take a finite index subgroup $H$ in $K$ to get $(B^K, H) \cong ((B^{[K:H]})^H, H) = (C^H, H)$. All in all $(A^G, H) \cong (C^H, H)$ where $|C| = |B|^{[K:H]} = (|A|^{[G:K]})^{[K:H]}$. This is potent at least if $[G:K] \geq 3$ and $[K:H] \geq 4$.
\end{proof}

\begin{corollary}\label{cor:strongryanscorollary}
Let $G$ be infinite f.g.\ residually finite group and let $(X, G)$ be a nontrivial full shift. Then there is a f.g.\ subgroup $H$ of $G$ such that $(X, H)$ satisfies strong finitary Ryan's theorem.
\end{corollary}

The point is that $(X, G)$ and $(X, H)$ have the same stabilized automorphism groups, in particular the f.g.\ group that realizes strong finitary Ryan's theorem is in the stabilized automorphism group of the original $(X, G)$.

We now give the proof. The proof of Lemma~\ref{lem:Aperiodic} contains a similar argument as the following. In what follows, the notation $x+y$ means:
\begin{enumerate}
\item
$(x + y)_g = x_g$ if $y_g = 0$
\item
$(x + y)_g = y_g$ if $x_g = 0$
\item
$x + y$ is undefined if $x_g \neq 0 \neq y_g$ for some $g$.
\end{enumerate}

\begin{lemma}
\label{lem:SumStab}
Suppose $x = x_1 + x_2$, $y = y_1 + y_2$, and suppose that
\begin{itemize}
\item the supports of $x_1, y_1$ are contained in the ball of radius $R$ at origin,
\item the supports of $x_2, y_2$ are contained in an annulus of inner radius $r$ and width $R$,
\item the supports of $x_1, y_1$ have equal cardinality,
\item the support of $x_1$ has larger cardinality than those of $x_2, y_2$, and
\item $r$ is much larger than $R$.
\end{itemize}
Then $gx = y \implies gx_1 = y_1 \wedge gy_2 = y_2$.
\end{lemma}

\begin{proof}
Assume $gx = y$ and the listed properties. It follows that $x_2$ and $y_2$ also have equal cardinality supports. It suffices to show that the left $g$-translation does not move any elements from the support of $x_1$ to the support of $y_2$. Namely then due to cardinalities it also does not move elements the support of $x_2$ to that of $y_1$.

But indeed if $h \in \supp(x_1)$ and $gh \in \supp(y_2)$, then all elements in $g \supp(x_1)$ have word norm at least $r-R$, and thus $g \supp(x_1) \subset y_2$, a contradiction since $x_1$ has cardinality larger than $y_2$.
\end{proof}

To give a hint how this is used, suppose $y_1 = x_1, y_2 = x_2$ and both are finite-support configurations, with $x_1$ having larger support than $x_2$. Then if we look at $x_1 + gx_2$ for large $g$, the assumptions hold. Now the support of $gx_2$ is even contained in a ball of radius $R$, far from the origin. But the lemma also applies to $x_1 + x_2g$, where the support of $x_2g$ may be sparsely spread on an annulus. (For $g \in G$, we write $x \mapsto xg$ for the automorphism defined by $xg_h = x_{hg}$.) We will use both cases (and $y_i$ not necessarily equal to $x_i$).

\begin{lemma}
\label{lem:Aperiodization}
Consider a full shift on a finitely-generated infinite group over alphabet $A \times B \ni (0, 0)$. Then there is a finitely-generated subgroup $H$ of the automorphism group such that the $H$-orbit of every aperiodic configuration contains a point where both tracks are aperiodic. Furthermore, the same element $h \in H$ will work for any point with sufficiently small support.
\end{lemma}

\begin{proof}
For now, ignore the last sentence, we'll return to it at the end.

As generators take partial shifts on the two tracks, and all symbol permutations. Consider $(x, y) \in (A \times B)^G$ with trivial stabilizer. By right shifting $x$ around, and applying symbol permutations, we may assume the support of $y$ is much larger than that of $x$, and of course we cannot add a period by applying automorphisms.

Then apply a right partial shift by $h$ to $x$ to get the configuration $xh$. Observe that $xh$ and $x$ have the same stabilizer.

First, suppose the simple case $A = B$. Then observe that we can literally add $y$ to the first track, to get the first track to contain the point $y + xh$. Now Lemma~\ref{lem:SumStab} applies and the only possible periods are $g$ such that $gy = y$ and $gxh = xh$, equivalently $g(x, y) = (x, y)$, which is impossible for $g \neq 1_G$ since the pair is aperiodic.

In general, when not necessarily $A = B$, we observe that any period of $y$ is a period of all characteristic functions of individual symbols in it. We can add such a characteristic function to the first track to get $\pi(y) + xh$ which by Lemma~\ref{lem:SumStab} will be only periods of $x$ and $\pi(y)$. Repeating this at most $|A|$ times (always increasing the support of the $B$-track between steps), we get that the $A$-track is aperiodic.

Next, to get the $B$-track to be aperiodic as well, increase its support to be very large (more than $|A|$ times the support of the $A$-track) by adding right shifts of the $A$-track to it. Then as in the previous paragraph, right shift the $A$-track far, add a projection to the $B$-track and apply Lemma~\ref{lem:SumStab}, repeat at most $|A|$ times to kill the stabilizer of the $B$-track.

Finally note that none of this discussion looked at the contents of $x, y$, we simply had to use sufficiently large translates. This explains the last sentence.
\end{proof}

We now recall the commutator trick in abstract form.

Let $X$ be a set, let $G$ be a group, and consider the product group $G^X$. For $A \subset X$, we write $g|A$ for the element $a \in G^X$ defined by $a_x = \begin{cases}
g & \mbox{if } x \in A \\
1_G & \mbox{otherwise.}
\end{cases}$.

\begin{lemma}[Commutator trick]
\label{lem:CommutatorTrick}
Let $X$ be a set, let $G$ be a group. Then
\[ [g|E, h|F] = [g, h]|{E \cap F} \]
for all $g, h \in G$, $E, F \subset X$.
\end{lemma}

\begin{proof}
If $x \in E \cap F$, obviously $[g|E, h|F]_x = [g, h]$. If $x \notin E \cap F$, then $(g|E)_x = 1_G$ or $(h|F)_x = 1_G$ so $[g|E, h|F]_x = 1_G$.
\end{proof}

We may think of $a \in G^X$ as something of a ``random'' group element, so that $g|E$ means ``apply $g$ if condition $E$ holds''. The lemma says that the commutator of two group elements, applied under respective events $E, F$, will apply the commutator when both events happen, and otherwise behaves as the identity. We will not explicitly explain what $X$ is in the following proof, but rather use this terminology of events.

\begin{lemma}
\label{lem:SymbolPermutations}
Consider a full shift on a finitely-generated infinite group over alphabet $A \times B \ni (0, 0)$ with 
$|B| \geq 5$. Then there is a finitely-generated subgroup $H$ of the automorphism group such that for any aperiodic finite configuration $x \in A^G$ and any even permutation $\pi$ of $B$, and any $R$, there is an automorphism in $H$ that applies the permutation $\pi$ to the symbol at the origin on the $B$-track when the configuration on the $A$-track is $x$, but fixes all points with support contained in the $R$-ball such that the first track does not contain a point from the orbit of $x$.
\end{lemma}

\begin{proof}
Take as generators of $H$ all partial shifts and all even symbol permutations.

If $\pi \in \Alt(B)$ and $C \subset B^G$ is clopen, then we define the automorphism $f_{\pi,C} \in \Aut((A\times B)^G)$ is defined at the origin by
\[ f_{\pi,C}((x, y))_{1_G} = \begin{cases}
(x_{1_G}, \pi(y_{1_G})) & \mbox{if } x \in C \\
(x_{1_G}, y_{1_G}) & \mbox{otherwise.}
\end{cases} \]

We show that $f_{\pi,C} \in H$ for all $\pi \in \Alt(B)$ and clopen $C$. It suffices to consider cylinders $C = [P]$, $P : D \to A$ where $D \subset G$ is finite. We observe that for each $a \in A$, $\pi \in \Alt(B)$ and $g \in G$, we have $f_{\pi,[a]_g} \in H$ by conjugating a generator by a partial shift.

Next, we observe that $|B| \geq 5$ implies that $\Alt(B)$ is a perfect group. (In fact, by a theorem of Ore, every element of this group is a commutator.) The result now follows by induction from the commutator trick because $[P]$ is a finite intersection of such $[a]_g$. In the lemma, we can take $X = A^G$, and consider clopen events $E, F$.

Now taking $C$ a very small neighborhood of $x$, we can perform any even permutation of the symbol under the origin of $x$, without affecting symbols under any other configuration with support contained in the $R$-ball, as desired.
\end{proof}

Next we recall the idea of universal gates from~\cite{SaloUniversal2022}.

\begin{lemma}[\cite{SaloUniversal2022}]
Let $B$ be a nontrivial finite alphabet with $|B| \geq 3$. If $n \geq 2$, then every even permutation of $B^n$ can be decomposed into even permutations of $B^2$ applied in adjacent cells. That is, the permutations
\[ w \mapsto w_0 w_1 \dots w_{i-1} \cdot \pi(w_i w_{i+1}) \cdot w_{i+2} \cdots w_{n-1} \]
are a generating set of $\Alt(B^n)$ where $\pi$ ranges over $\Alt(A^2)$, and $i$ ranges over $0, 1, 2, \dots, n - 2$.
\end{lemma}

The exact same proof can be used to prove the following.

\begin{lemma}
Let $B$ be a nontrivial finite alphabet with $|B| \geq 3$. Let $(V, E)$ be any connected graph. Then the permutations of $B^V$ that apply even permutations of $\Alt(B^e) \cong \Alt(B^2)$ at edges $e \in E$ generate $\Alt(B^V)$.
\end{lemma}

\begin{lemma}
Consider a full shift on a finitely-generated infinite group over alphabet $A \times B \ni (0, 0)$ with $B = (B')^k$ for $k \geq 2$, and $|B'| \geq 3$. Then there is a finitely-generated subgroup $H$ of the automorphism group such that for any aperiodic finite configuration $x \in A^G$ and any three arbitrary finite configurations $y_1, y_2, y_3 \in B^G$, and any $R$, there is an automorphism in $H$ that performs the rotation $((x, y_1), (x, y_2), (x, y_3))$ without moving any other point with support contained in the $R$-ball.
\end{lemma}

\begin{proof}
We suppose for simplicity that $k = 2$. The argument generalizes easily to other $k$ (and this is needed to cover some potent alphabets).

Again take as generators the partial shifts and symbol permutations, but this time also take a partial shift of the first $B'$-track. By Lemma~\ref{lem:SymbolPermutations}, we can perform any even symbol permutation under an individual aperiodic configuration $x$ without affecting any other symbols under it, nor symbols under any other configuration (with small enough support). We can conjugate these maps by partial shifts to permute the contents of any cell under a shift of $x$ (again without affecting the $B$-track on other configurations).

We can also conjugate by a partial shift on one of the $B'$-tracks of the $B$-track, to apply a permutation on $(a, b) \in (B')^2$ where $a$ and $b$ now come from different adjacent group elements. Consider now the graph with two nodes at each group element, and with edges between the nodes at the same element, and edges between the top node of a cell and the bottom node of its neighboring cell. We can perform any even permutation of the cells in any such edge. Therefore, we can perform any even permutation of the entire configuration.
\end{proof}

Next, we recall the hypergraph lemma. Let $(V, E)$ be a finite $3$-hypergraph, meaning $E$ is a set of $3$-subsets of $V$. Associate to it the graph with edges the $2$-subsets of hyperedges in $E$. If this resulting graph is connected, then $(V, E)$ is called weakly connected.

\begin{lemma}
Let $(V, E)$ be a weakly connected finite $3$-hypergraph. For each $e \in E$, let $\pi_e$ be any $3$-rotation of the vertices of $e$, seen as an element of $\Sym(V, E)$. Then $\langle \pi_e \;|\; e \in E \rangle = \Alt(V)$.
\end{lemma}


\begin{lemma}
\label{lem:PermutingConfigs}
Consider a potent full shift on a finitely-generated infinite group. Then there is a finitely-generated subgroup $H$ of the automorphism group such that for any $R, r, H$ can perform any even permutation of the set of orbits of aperiodic configurations which have a translate with support in the $R$-ball, while fixing all other points with support contained in the $r$-ball.
\end{lemma}

\begin{proof}
We consider an alphabet $A \times B \ni (0, 0)$ with $A = A' \times A', B = B' \times B'$ and $|B'| \geq 3$ (it helps to have two tracks in different roles). With $|A'| = |B'|$ this covers all potent squares, but not all potent cases. However, the argument adapts to $B$ easily.

Again take partial shifts and symbol permutations as generators. We have already seen that if $x \in A^G$ is aperiodic, then we can perform any $3$-rotation on the second track. Similarly, if $y \in B^G$ is aperiodic, symmetrically we can perform any $3$-rotation on the first track. An easy application of the hypergraph lemma shows that we can perform any $3$-cycle of configurations $(x_1, y_1), (x_2, y_2), (x_3, y_3)$ where in each pair, at least one of the $x_i, y_i$ is aperiodic.

Finally, we can conjugate any $3$-cycle $(x_1, y_1), (x_2, y_2), (x_3, y_3)$ where the pair configurations $(x_i, y_i)$ are aperiodic, to configurations where both tracks are individually aperiodic, using Lemma~\ref{lem:Aperiodization} (specifically we use the last sentence here).

Note that at each step, if $(x, y)$ has elements in its $r$-ball but is not one of the configurations that interest us, then the basic $f_{\pi,C}$ we used to build everything from act trivially (if we use small enough $C$ as a function of $r$).
\end{proof}

\begin{theorem}
Consider a potent full shift on a finitely-generated infinite group 
Then the group $\PAut(S)$ acts infinite-orbit transitively on the aperiodic homoclinic points. 
\end{theorem}

\begin{proof}
Again for simplicity assume the alphabet is $S = A \times A$ with zero symbol $(0,0)$ (for higher powers of $A$ the argument adapts easily, just ignore the other tracks).

Note that all the generators used in previous lemmas are various partial shifts and symbol permutations, so in $\PAut(S)$. By Lemma~\ref{lem:PermutingConfigs} we can perform any even permutation on the configurations up to a shift.

We still need to fix the offsets of configurations. We show that we can perform a left shift [sic] in direction $s$ on one configuration, without affecting any other homoclinic configuration of interest. We can then conjugate this configuration to another so that any particular configuration we are interested in is shifted by $s$, while others are fixed.

We will actually have to shift another configuration by $s^{-1}$ as a side-effect, but we can conjugate this to a homoclinic point with large support which we are not interested in.

For this, identify finite subsets $S \subset G$ with configurations over some particular nonzero symbol $1$. We can now easily swap $(\{1\}, \{1\})$ and $(\{1\}, \{s\})$ for a generator $s$ without affecting the other configurations with small supports, and fixing the $A$-track (this is not an even permutation, but simply perform another swap somewhere else, with a configuration we are not interested in). To do this, simply apply the above construction again, i.e.\ this is nothing but a specific conditional permutation of the second track conditioned on the contents of the first track.

Then swap $(\{1\}, \{s\})$ and $(\{s\}, \{s\})$ similarly. The composition moves $(\{1\}, \{1\})$ to $(\{s\}, \{s\})$, and $(\{1\}, \{s\})$ to $(\{s^{-1}\}, \{1\})$ and does not move other configurations that interest us.
\end{proof}

One can now conclude the following using Proposition~\ref{prop:WeakRyans}:

\begin{proposition}
The weak finitary Ryan's theorem holds for potent full shifts. 
\end{proposition}

The only non-trivial thing to check is that weak Ryan's theorem holds. It is a relatively simple exercise to prove this for all full shifts, but for readers who prefer to avoid this exercise, one may look at the proof of Proposition~\ref{prop:WeakRyans} to discover that we only use weak Ryan's theorem to discount permutations of fixed points. For a potent full shift, and more generally any subshift with at least $3$ fixed points, these are automatically discounted as soon as our finitely generated subgroup contains all symbol permutations. In any case, the proposition is weaker than strong finitary Ryan's theorem, our main goal.

\begin{lemma}
\label{thm:FullShiftPropertyP}
For any potent power full shift $(X, G)$, there is a f.g.\ subgroup of the automorphism group with property P.
\end{lemma}

\begin{proof}
Let $x$ be an aperiodic homoclinic, and $y$ not in its shift orbit. Applying ($1$-)orbit transitivity on aperiodic homoclinics we may assume $x$ has just a single symbol $1$ at the origin, on the first track. If $y$ has a nonzero symbol on one of the tree bottom tracks, then we can use a symbol permutation to affect $y$ without affecting $x$.

Suppose then that $y$ also has its support contained fully in the first track. Then apply a full-support even symbol permutation on the second track, conditioned on seeing some large pattern that appears in $y$ but not in $x$. Note that such a pattern must indeed exist since $y$ is not in the shift-orbit of $x$, and $x$ is a homoclinic point.
\end{proof}

We are now able to prove Theorem~\ref{thm:FullShiftStrongRyan}.

\begin{proof}[Proof of Theorem~\ref{thm:FullShiftStrongRyan}]
By the previous theorem, there is a f.g.\ subgroup $H$ of the automorphism group with property P. Obviously a full shift is a pointed shift of finite type with dense aperiodic homoclinic points, and $G$ acts faithfully, so by Theorem~\ref{thm:StrongRyan} the centralizer of $H$ in $\Homeo(X)$ is $G$.
\end{proof}

\subsection{Glider diffusion and corner decreasability}

In the next section, we show a variant of the previous argument which covers a reasonable class of SFTs on $\Z^d$. Commutator tricks work fine on all SFTs where we have a product structure (even somehow locally), but there is no reason why one should be present.

Thus our plan is as follows: We show that we can conjugate any finite set of homoclinic points to a tuple of homoclinic points with very sparse supports, by ``diffusing them into clouds of gliders''. Then we show that such clouds of gliders can be interpreted locally as having a product structure -- concretely, they can be seen as belonging to a certain product subshift, by using the embedding theorem of Meyerovitch. These main feats are performed in the present section, and then it is easy to finish in the following section.

The arguments of this section are for $\Z^d$, and we use at least the facts that this group has no nonamenable subgraphs, and is orderable (and has a copy of $\Z$, but this is automatic). It might be possible to directly generalize the arguments for such groups, covering at least the Heisenberg group. To cover a larger class of groups, new ideas may be needed.


Define $\eta_{\vec v,i}$ to be the homoclinic point with support $\{\vec v\}$ and symbol $i$ in this position. Also $\eta_i = \eta_{\vec v, i}$.


We recall the finite extension property introduced in~\cite{BricenoMcGoffPavlov2018}. Suppose that $X$ is an SFT. We say $X$ has the \emph{finite extension property} (FEP) if there is a finite set of forbidden patterns $\mathcal{F}$ defining $X$ and $r \in \mathbb{N}$ satisfying the following: if a pattern $U$ defined on $S \subset \mathbb{Z}^{d}$ can be extended to a pattern on $S+[-r,r]^{d}$ which does not contain any pattern from $\mathcal{F}$, then there exists $x \in X$ such that $x|S = U$.

We will use the lexicographic order on $\mathbb{Z}^{d}$ defined inductively (using the usual order on $\mathbb{Z}$) by $(u_1, \ldots, u_d) \leq (v_1, \ldots, v_d)$ if $u_1 < v_1$, or $u_1 = v_1$ and $(u_2, \ldots, u_d) \leq (v_2, \ldots, v_d)$.

Recall as defined right before Lemma~\ref{lem:SumStab}, $+$ denotes the union of two configurations or patterns which have nonzero supports non-intersecting.

\begin{lemma}
Let $X \subset A^{\Z^d}$ be a subshift, and let $x_1, \ldots, x_n$ be nonzero finite-support points with different orbits. Then there is a conjugacy $\pi : X \to Y$ such that $\pi(x_i)$ has nonzero support $\{\vec 0\}$.
\end{lemma}

\begin{proof}
Define first an auxiliary map to alphabet $A \times \{0, 1, \ldots, i\}$, which marks the radius-$R$ central segment of $x_i$ by writing $i$ on the second track, and otherwise writes 0 on the second track. If we pick $R$ large enough, this map takes $x_i$ to $(x_i, \eta_{\vec 0, i})$.

Next, pick an abelian group structure on $A$ with zero element $0$. Define an automorphism of the full shift on $A \times \{0, 1, \ldots, i\}$, which subtracts $x_i$ on the first track whenever the second track contains $\eta_{\vec 0, i}$.

The composition of the two morphisms has the desired property, since now $x_i$ maps to $(0^{\Z^d}, \eta_{\vec 0, i}) = \eta_{\vec 0, (0, i)}$.
\end{proof}

Now we construct the f.g.\ subgroup of an automorphism group to perform ``glider diffusion'' which means we turn the support into a cloud of gliders. Recall that a glider (or spaceship) for an automorphism is just a finite-support point which the automorphism maps into a distinct point in its shift-orbit. A ``cloud of gliders'' informally means a configuration which is a union of gliders with sufficiently disjoint supports so that they move independently.

\begin{lemma}
Let $x_1, x_2 \in X$ be two configurations with supports $\{0\}$ in an SFT $X$ with radius $R \geq 2$. Let $f$ be the morphism $f : X \to A^{\Z^d}$ that exchanges their central symbols $a_1$ and $a_2$ if they are at distance at least $R$ from other parts of the support. Then $f$ codomain-restricts to an automorphism of $X$ of order $2$.
\end{lemma}

\begin{proof}
This is obvious.
\end{proof}

\begin{lemma}
\label{lem:g}
Let $x_1, x_2 \in X$ be two configurations with supports $\{0\}$ in an SFT $X$ with radius $R \geq 2$, and central symbols $a_1, a_2$ respectively. Let $g$ be the morphism $g : X \to A^{\Z^d}$ which exchanges the patterns $0a_1$ and $a_20$ (as patterns on $\{\vec 0, e_1\}$) when the $a_i$-symbols are at distance at least $2R$ from any other nonzero symbol, and are at distance at least $3R$ from the occurrence of one of the patterns $0a_1, a_20$. Then $g$ codomain-restricts to an automorphism of $X$ of order $2$.
\end{lemma}

\begin{proof}
Consider the gate $\chi$ that swaps $0a_1$ and $a_20$ when the $a_i$-symbol is at distance at least $2R$ from any other nonzero symbol, and is at distance at least $3R$ from the occurrence of one of the patterns $0a_1, a_20$. We claim that $\chi$ and $\chi^{\vec v}$ commute for any vector $\vec v$.

It suffices to look at a case of $$[\chi^{\vec v}, \chi](x)$$ where the first application of $\chi$ is actually applied. Then the application of $\chi^{-1}$ will later also succeed, whether or not $\chi^{\vec v}$ performs a rewrite in the next step. Namely, $\chi^{\vec v}$ cannot possibly perform a rewrite that affects the positions at distance at most $2R$ because the swaps would have be to close by and would have both been cancelled. Rewrites further away can change the support by one in Hausdorff distance, but do not affect the positions where the patterns $0a_1$ or $a_20$ occur, and these cancellations trump the effect of changing the support. So $\chi^{-1}$ cancels the first swap, and similarly $(\chi^{\vec v})^{-1}$ cancels the swap at $\vec v$.

It follows that $\chi^{\Z^d}$ is a gate lattice, thus an automorphism. This is just another description of $g$.
\end{proof}

\begin{lemma}
The points $x_1, x_2$ are gliders for $h = g \circ f$. Furthermore, $h$ preserves the cardinality of the support.
\end{lemma}

\begin{proof}
Both claims are obvious.
\end{proof}

It is easy to understand the dynamics of $h$: each $a_1$ symbol sufficiently separated from other symbols moves to the right (meaning in direction $e_1$) and $a_2$-symbols move left (in direction $-e_1$), and when they get too close to other gliders or parts of the support, they change their type and bounce back. This can be expressed as follows.

\begin{lemma}\label{lem:Dep}
Let $h, R$ be as above. Then for every finite-support configuration $x$, there exist $n$ and $p \geq 1$ such that $h^n(x) = x_L + x_M + x_R$ and for all $k$ we have
\[ h^{n+kp}(x) = \sigma^{-kpe_1}(x_L) + x_M + \sigma^{kpe_1}(x_R). \]
Furthermore, the support of $x_L$ is entirely contained $R$ steps left of the support of $x_M$, which in turn is $R$ steps right of the support of $x_R$. Furthermore, $x_L$ is a cloud of translates of $a_2$-symbols with separation at least $R$, and $x_R$ is a cloud of $a_1$-symbols with separation at least $R$.
\end{lemma}

\begin{proof}
By the construction of $h$, the only thing that happens in its action is that gliders, which are separated $a_1$- and $a_2$-symbols, ``move'' either left or right, and sometimes change to each other. There is a well-defined notion of which glider in the preimage corresponds to which glider in the image, and the definitions of $f$ and $g$ allow us to track the movement and type of a single glider in a natural way.

Suppose some glider is not bounded to a finite region. Then there can be gliders arbitrarily far to the left or to the right in the iteration of $h$, suppose on the left. Then it must be an $a_2$-symbol, traveling at constant speed. The movement of such glider cannot be stopped by any other glider (and other parts of the support cannot move). Thus, inductively the gliders that do not stay bounded eventually in a natural sense ``join'' the $x_L$-part or symmetrically the $x_R$-part. The remaining part that stays bounded must behave periodically since the rule is bijective, giving us $n$ and $p$.
\end{proof}

Next, we need to discuss the gluing properties that allow glider diffusion.

\begin{definition}
Let $X$ be a subshift with a zero point. A point is in \emph{good position} if it is nonzero and the lexicographically minimal position in its support is the origin.
\end{definition}

Note that not all points can be shifted to good position (e.g.\ a fixed point not containing the zero-symbol at all), that every point has at most one point in its orbit which is in good position, and that every finite-support point can be moved to good position.

\begin{definition}
A subshift $X$ has the \emph{corner decreasability property} if there exists $R$ such that whenever $x \in X$ is in good position and has finite support, there is another point $y$ whose nonzero support has strictly smaller cardinality than that of $x$, and which differs from $x$ only in the $R$-ball around the origin.
\end{definition}

\begin{lemma}
Every FEP subshift has the corner decreasability property.
\end{lemma}

\begin{proof}
Consider the local connected component of the origin in a configuration in good position. Since $\Z^d$ has polynomial growth, we can look at balls for this support's intrinsic metric, and at some point the boundary is at most $\epsilon$ times the size of the support seen so far (what we really use is that $\Z^d$ does not contain a nonamenable graph), and we have a bound on when this happens. One particularly nice possibility is that the component is actually finite, and we can just erase the entire component. Otherwise, rewrite the component with zeroes.

Let $S$ be the support seen so far, $B$ the boundary of it in the intrinsic metric. Let $C$ be the radius $100 R$-ball around $B$. Replace $S$ with zeroes, and restrict the resulting configuration to $\Z^d \setminus C$. Due to the thickness of $C$, this configuration allows a locally valid $R$-extension, since we can use $0$s to extend on one side of $C$, and the original configuration to extend on the other side of $C$. Therefore this extends to a legal configuration. The extension can only add $|C|$ elements to the support, and we have removed $|S|$. We have $|C| \leq 100 (2d+1)^R |B| < |S|$ for small enough $\epsilon$, so the new configuration has smaller support.
\end{proof}

\begin{lemma}
Let $X$ be a pointed SFT with radius $R$ which has the corner decreasability property, and let $p_i = \eta_{\vec 0, a_i}$ for $i = 1, 2$ be distinct. Then there is a finitely-generated subgroup $H$ of its automorphism group such that for any tuple of finitely-supported points $x_1, \ldots, x_n$ there is $k \in H$ such that each $kx_i$ has $R$-separated support over alphabet $\{0, a_1, a_2\}$, and there are no further glider collisions in the forward orbit of the automorphism $h$ constructed above.
\end{lemma}

\begin{proof}
The plan is to use the corner decreasability property to increase the number of gliders in at least one of the points $x_i$. Then we can apply a massive power of $h$ to move this glider away from the support (and possible also move other gliders away, possibly also on other $x_j$s), and continue until there are only gliders present.

Consider some $x_i$. We suppose we have already applied a huge power of $h$, and we may ignore the gliders near the origin. Thus we may assume $x_i$ does not have any free gliders, and the lexicographically minimal position in its support is fixed by $h$ (though there may be some bouncing gliders inside the support of course).

By the corner decreasability property, there is some $y_i$ with smaller support to which we can legally change $x_i$ at the corner. Change further $y_i$ to $z_i$ by inserting $p_2$ type gliders (so $a_2$ symbols) in random positions to the left (direction $-e_1$) of the support, so that $x_i$ and $z_i$ have the same support cardinality.

We note that a lexicographically minimal position in a configuration is a pattern that cannot intersect a small translate of itself. If the pattern of $p_2$-symbols is random, then with high probability we can exchange central parts of $x_i$ and $z_i$ by an automorphism, so that the support size of every configuration is preserved. Near the actual lexicographically minimal position of $x_i$, a glider will actually then escape when applying a large power of $h$, and we have effectively decreased the total support of the $x_i$'s (more precisely, of their $x_M$-parts in Lemma~\ref{lem:Dep}).

The radius of automorphisms $k$ needed for this scheme is just a function of the corner decreasability radius and enough so that shifts of $x_i$ and $z_i$ are not confused locally. So we can realize this scheme with a finitely generated subgroup $H$ (containing $h$ and the automorphisms exchanging central parts of $x_i, z_i$).
\end{proof}

\subsection{Finitary Ryan's over $\mathbb{Z}^{d}$}

We now prove a finitary Ryan's theorem for corner decreasable SFTs of finite type with a fixed point, working over the group $\mathbb{Z}^{d}$. This will be key to our use of local $\mathcal{Q}$ entropy later.

\begin{theorem}
\label{thm:FinitaryRyan}
If $X$ is a contractible $\Z^d$-SFT with a fixed point, then $(X,\mathbb{Z}^{d})$ satisfies strong finitary Ryan's theorem.
\end{theorem}

For $d = 1$ (and a larger class of subshifts), and at least for the weak finitary Ryan's theorem, this has been shown by Kopra~\cite{Kopra2020}. Kopra also optimizes the number of generators to the optimal number two, whereas we make no such attempt. Recall that corner decreasability follows from the FEP property.

In the theorem, we assume for simplicity that we have a fixed point, say the zero point $0^{\Z^d}$, since in our applications in the context of stabilized automorphism groups, we can turn any periodic point into a fixed point by passing to a finite index subgroup. The proof does adapt, however, to $\Z^d$-SFTs with a totally periodic point, and contractible SFTs have periodic points (see the next lemma). 

We say a subgroup of the automorphism group of a subshift $X$ acts $k$-orbit-transitively on a set $Y \subset X$ if it can map any $k$-tuple of points from $Y$ of which no two are equivalent up to shift, to any other such tuple. The main point of the proof is to prove the following result of independent interest.

\begin{theorem}
\label{thm:OrbitTransitivity}
If $X$ is a contractible $\Z^d$-SFT with a fixed point, then the automorphism group has a finitely-generated subgroup $G$ which acts $k$-orbit-transitively on the nontrivial homoclinic points, for all $k$.
\end{theorem}

Here also we could use a periodic point instead of the fixed point, but the proof would need to be adapted.

Following~\cite{MeyerovitchEmbeddingTheorem}, we say a $G$-subshift $X$ has the \emph{map extension property} if the following holds: whenever $Y$ is any $G$-subshift such that for every $y \in Y$ there exists $x \in X$ such that $\textrm{stab}(y) \subset \textrm{stab}(x)$, then for any (possibly empty) subshift $Z \subset X$ and map  $\pi \colon (Z,G) \to (Y,G)$, there exists a map $\overline{\pi} \colon (X,G) \to (Y,G)$ which extends $\pi$.

\begin{lemma}[\cite{PoirierSaloContractible}]
\label{lem:EveryContractible}
Every contractible $\Z^d$-SFT has
\begin{itemize}
\item a periodic point,
\item the map extension property of Meyerovitch,
\item the finite extension property of Brice\~no, McGoff and Pavlov.
\end{itemize}
\end{lemma}

Contractible $\Z^d$-SFTs are in fact precisely equal to the subshifts with map extension property of Meyerovitch~\cite[Theorem 1.8]{PoirierSaloContractible}. Contractible $\Z^d$-SFTs with a fixed point, in turn, are precisely the retracts of full shifts in the category of SFTs~\cite[Theorem 1.5]{PoirierSaloContractible}.

We start by showing that the group itself always acts $k$-orbit transitively (on SFTs). The trick is the same as the construction of $h$ in the previous section. This is not directly used in the proof.

\begin{lemma}
\label{lem:OrbitTransitivity}
Let $X$ be any SFT with a fixed point. Then the automorphism group acts $\infty$-orbit-transitively on the homoclinic points.
\end{lemma}

\begin{proof}
If $X$ is a singleton, then the claim is trivial. Note that if X has at least one homoclinic point, then since it is an SFT it has infinitely many shift orbits of homoclinic points. We first observe that it suffices to show that for any finite set of homoclinic points in distinct orbits, we can perform any permutation of it, and we can move a single homoclinic point in its orbit without moving the others.

The first can be done easily, the automorphism just permutes the local central patterns as long as it sees a sufficient amount of 0's around, say at least $M$ zeroes in each direction. There is a subtlety: the permutation itself can change the set of positions where 0's appear for another appearance of the central patterns of the homoclinic points. Therefore, in addition to cancelling the permutation when a nonzero symbol appears at distance less than $M$, we should cancel it when one of the central patterns appears at distance $N$ which is much larger than $M$. The permutation does not change the positions where central patterns appear, and it does not change the supports enough to change which permutations are cancelled, so this is an automorphism.

The second can now also be done easily. Say we want to move $x$ to the right (where again to the right means in the direction of $e_{1}$). Then we simply pick another homoclinic point $y$ not on the list, and we move $x$ right and $y$ left, and swap $x$ and $y$ if we cannot. This can be done by using the construction in the previous paragraph, by first performing the permutation $(x, ...) \leftrightarrow (y, ...)$ and then performing the permutation $(\sigma_{e_{1}}(x), ...) \leftrightarrow (y, ...)$.
\end{proof}

We note that this proves the following.

\begin{corollary}
Let $X \subset A^{\Z^d}$ be any SFT such that for some f.g.\ $H \leq \Aut(X)$, $C_{\Aut(X)}(H) \leq \Aut_0(X)$. Then weak Ryan's theorem holds for $X$. In particular, it holds for any SFT where the action of $\Aut(X)$ on the fixed points has trivial center.
\end{corollary}

\begin{proof}
By Proposition~\ref{prop:WeakRyans} the above f.g.\ subgroup $H$ of $\Aut(X)$ satisfies $C_{\Aut_0(X)}(H) = \Z^d$ (i.e. the group of shifts). To get all of $\Aut(X)$ in the subscript, recall that in the proof of the second part of Proposition~\ref{prop:WeakRyans} the full weak Ryan's theorem is only used only to force the centralizer to be inside $\aut(X)$. Here this is directly in the assumptions.
\end{proof}

Note that the assumption in the statement about the action on fixed points holds for example if there is a unique fixed point, or if $\Aut(X)$ acts trivially on the fixed points, or if there are at least three fixed points on which $\Aut(X)$ acts by the symmetric group.

We next state the Meyerovitch Embedding Theorem. Here $X_G$ is the set of points with stabilizer containing $G$ when $G$ is not of finite index, and for finite-index $G$ it is the points with exactly stabilizer $G$. The function $h$ is entropy, and $\Ker$ is the global stabilizer.

\begin{theorem}[Theorem 1.4 of~\cite{MeyerovitchEmbeddingTheorem}]
Let $G$ be a countable abelian group. Suppose $Y, X$ are $G$-subshifts, and $X$ has the map extension property. Further let $Z \subset Y$ be a subshift, and let $\hat\rho : Z \to X$ be an embedding. Then $\hat\rho$ extends to an embedding $\rho : Y \to X$ if and only if for every subgroup $G_0$ of $G$, one of the following conditions holds:
\begin{itemize}
\item $\hat\rho_{Z_{G_0}} : Z_{G_0} \to Y_{G_0}$ extends to an isomorphism $\rho_{G_0} : X_{G_0} \to Y_{G_0}$, and
\item $h(Y_{G_0}) < h(X_{G_0})$ and $\Ker(X_{G_0}) \leq \Ker(Y_{G_0})$.
\end{itemize}
\end{theorem}

We now prove Theorem~\ref{thm:OrbitTransitivity}.

\begin{proof}[Proof of Theorem~\ref{thm:OrbitTransitivity}]
Contractibility implies FEP and thus the arguments of the previous section go through and we can turn any finite set of homoclinic points in distinct orbits into two $R$-separated clouds of $a_1$ and $a_2$-symbols, for any $R$.

By taking a larger $R$, we may assume that actually in $X$ the symbols can appear with smaller separation than is used by the clouds, which allows us to apply another automorphism and Lemma~\ref{lem:OrbitTransitivity} to rewrite all the $a_2$-symbols into a pair of $a_1$-symbols. Thus, we may assume our clouds finally use just one nonzero symbol, and still the separation between individual symbols can be as large as we want.

Using an even larger separation, we get that there are exponentially many patterns over the symbol $a_1$ that appear in $X$ by do not appear in the clouds. We can recode such patterns into symbols so that ultimately we may assume that there are exponentially many symbols $a_i$ (as a function of $R$)< that can freely appear with $R$-separation, while the cloud diffusion process (together with the previous paragraph) still produces a cloud of $R$-separated symbols $a_1$.

Next, by a fancier recoding we will change the alphabet of $X$ to $\{0, 1, 2, 3, 4, 5\}^3 \cup \Sigma$ where $(0,0,0)$ represents the original zero-symbol, $\eta_{(1, 0, 0)}$ corresponds to $a_1$, and $\Sigma$ is a large auxiliary alphabet. For this, we apply the Meyerovitch Embedding Theorem. Let $Y$ be the subshift over alphabet $\{0, 1, 2, 3, 4, 5\}^3$ where two nonzero symbols on the same track cannot appear at distance less than $R$ from each other. Let $Z$ be the subsystem where the two bottom tracks are zero, and the first track can only contain $0, 1, 2$. The subshift $Z$ can be directly seen as a subsystem of $X$ by projecting to the first track, and we take this embedding as $\hat\rho$.

We check the embedding conditions. Let $G_0$ be any subgroup of $\Z^d$. Then $h(Y_{G_0}) < h(X_{G_0})$, because by a greedy procedure, going track by track, we can find a uniformly $3R$-dense set of points in the support of $Y$ which is $R$-separated, and which is $G_0$-periodic. As discussed, in $X$ we have plenty of symbols we can use to code everything in this sparse set, together with a some additional bits to get excess entropy. (Here note the subtlety that if the index is finite then entropy means just counting, and in this case $X_{G_0}$ is the set of points with exactly stabilizer $G_0$.)

Then we show $\Ker(X_{G_0}) \leq \Ker(Y_{G_0})$, namely suppose $\vec v \in \Ker(X_{G_0})$. Then in particular $\vec v$ is a period of every point in $X$ which only uses the symbols $0, a_1, a_2, a_3, a_4, a_5$. But any $G_0$-periodic point of $Y$ will be have such a configuration on each of its tracks, so every $G_0$-periodic point of $Y$ must also have period $\vec v$.

We conclude that $Z \subset X$ extends to an embedding of $Y$ into $X$. After a recoding we may assume that $Y$ is contained in $X$, and this does not affect the clouds we have diffused configurations into, which are in $Z$.

By the cloud diffusion process, and the embedding argument above, we have now reduced the problem to the following: Given a tuple of homoclinic points in $Y \subset X$ (where again $Y$ is the set of configurations over the subalphabet $\{0,1,2,3,4,5\}^3$ with SFT rule that on each track the nonzero support is $R$-separated), which all have nonzero symbols only on the first track (and the only nonzero symbol used there is $1$), and which are in different shift-orbits, we have to show that we can perform any permutation of them or shift exactly one of them, by an automorphism of $Y$ which extends to an automorphism of $X$.

This can be done as follows. First, copy the first track of every point to the third track. Of course, we need to do this so that it preserves $Y$ (and further extends to $X$), but we can simply take the symbol permutation that swaps the symbols $(1,0,0)$, $(1,0,1)$ by an automorphism that only applies when there is sufficient distance to other symbols on the first and third track (this can be formalized for example by using a gate lattice as in Lemma~\ref{lem:g}. By also cancelling the application near forbidden patterns of $X$, we get this behavior with an automorphism of $X$.

Next, we observe that we can freely shift any individual track of $Y$, we call these \emph{partial shifts}. To extend this to an automorphism of $X$, we can redo the embedding argument with $Y$ replaced by the $R$-separated SFT over $\{0,1,2,\ldots,10\}^3$, so that each nonzero symbol $i\neq 0$ on each track has an ``antiparticle'' $11-i$. Now as in the definition of $h$ in the previous section, we can use trackwise symbol permutations and a gate lattice to move these particles around. (The antiparticles do not actually appear in the configurations we are interested in, and we can now forget about them.)

Next, we permute nonzero symbols of the third track by even permutations, applying the commutator trick. More precisely, for each even permutation $\pi$ of $\{1, 2, 3, 4, 5\}$ and each clopen set $C \subset \{0, 1, 2, 3, 4, 5\}^{\Z^d}$, we inductively produce an automorphism $\pi|C$ that performs the permutation $\pi$ on the nonzero symbols of the third track, if the configuration on the first track is in the clopen set $C$. This is exactly analogous to Lemma~\ref{lem:CommutatorTrick} and its use in Section~\ref{sec:FullShifts}.

In particular, we can pick $C$ and $\pi$ so that after applying the automorphism, under exactly one of the homoclinic points, exactly one of the $1$-symbols is changed to $2$, while all other points stay fixed. Next, we remove all the $1$-symbols from the third track and turn them to $0$, by repeating the symbol permutation from the first step. Note that the unique $2$ on the third track stays there as the symbol $(1, 0, 2)$ is not affected by this symbol permutation, and there are no additional side-affects (as far as the finite set of homoclinic points we are considering go).

This effectively marks a particular position of one particular homoclinic by a $2$-symbol on the third track. We can now use this marker, or ``Turing machine head'', to write the image of the homoclinic on the first track to the second track, by moving the head around and using symbol permutations that only affect the configuration near the head. Specifically, we use partial shifts on the third track, and the symbol permutations $(0, 0, 2) \leftrightarrow (0, 1, 2)$ and $(1, 0, 2) \leftrightarrow (1, 1, 2)$ (we pick the first one when we are writing $1$ in a position where there is no $1$ in that position on the first track). These symbol permutations are again done only in contexts where they don't produce forbidden patterns of $Y$ as in Lemma~\ref{lem:g}.

Finally, we can remove the head from the third track the same way we introduced it, ignoring the second track. We now do this one by one for all the finitely many homoclinic points, to end up with configuration $(y_i, z_i, 0^{\Z^d})$, where $z_i$ is the image of $y_i$ under the permutation we want to perform. Next, we observe that a similar argument can be used to turn $(0^{\Z^d}, z_i, 0^{\Z^d})$ into $(y_i, z_i, 0^{\Z^d})$, by applying the inverse of the map we are implementing. Now, applying the inverse of this transformation to $(y_i, z_i, 0^{\Z^d})$, we end up with $(0^{\Z^d}, z_i, 0^{\Z^d})$, and finally it suffices to swap the first and second track. This last step can again be done by a symbol permutation and thus extends to $X$.
\end{proof}

We conclude property P analogously as in Section~\ref{sec:FullShifts}.

\begin{lemma}
If $X$ is a contractible $\Z^d$-SFT with a fixed point. Then $\Aut(X)$ has an f.g.\ subgroup with property P.
\end{lemma}

\begin{proof}
Let $H$ be the group constructed in this section, which is $\infty$-orbit transitive. Let $x$ be an aperiodic homoclinic, and $y$ not in its shift orbit. Applying ($1$-)orbit transitivity on aperiodic homoclinics we may assume $x$ has just a single symbol $a_1$ at the origin, on the first track (in the sense that clouds of gliders can be considered to live in the subshift $Z$ obtained from the embedding theorem). If $y$ has an area which is locally in $Z$ and which has a nonzero symbol on one of the three bottom tracks, then we can use a symbol permutation to affect $y$ without affecting $x$.

Suppose next that $y$ does not have any large areas where it is in $Z$. Then clearly we can move $x$ without affecting $y$.

Suppose then that $y$ does have such areas, and in all these areas it has its support contained fully in the first track. Then we first copy the first track to the second (to obtain some nonzero symbols on the second track of both $x$ and $y$), and apply a nontrivial even symbol permutation on the second track's nonzero symbols, conditioned on seeing some large pattern that appears in $y$ but not in $x$. This affects $y$ but not $x$. Note since $y$ is not in the shift-orbit of $x$, and $x$ is a homoclinic point, we can indeed find a large enough pattern to separate them.
\end{proof}

Finally we can now prove Theorem~\ref{thm:FinitaryRyan}.

\begin{proof}[Proof of Theorem~\ref{thm:FinitaryRyan}]
By the previous theorem, there is a f.g.\ subgroup $H$ of the automorphism group with property P. Dense homoclinic points is immediate from contractibility and since $\Z^d$ is torsion-free they are aperiodic. By Theorem~\ref{thm:StrongRyan} the centralizer of $H$ in $\Homeo(X)$ is $G$.
\end{proof}

\section{Ghost centers}
Our goal now is to prove a promotion theorem, namely Theorem~\ref{thm:isoupgrade}, that upgrades any abstract isomorphism of stabilized automorphism groups to a pointed one. For this we will use the notion of ghost centers, introduced in~\cite{schmiedingLocalMathcalEntropy2022}.

Let $G$ be a group. We define a \emph{ghost center} of $G$ to be a subgroup $H \subset G$ such that for every finitely generated subgroup $K \subset G$, there exists a finite index subgroup $J \subset H$ such that $J \subset C_{G}(K)$. Given in addition another group $E$, by an \emph{$E$-ghost center} of $G$ we mean a ghost center in $G$ which is abstractly isomorphic to $E$.

It is clear that if $\Psi \colon G_{1} \to G_{2}$ is an isomorphism of groups, $E$ is a group, and $H \subset G_{1}$ is an $E$ ghost center of $G_{1}$, then $\Psi(G_{1})$ is an $E$-ghost center of $G_{2}$.

We say $H \subset G$ is a \emph{weak ghost center} if for every finitely generated subgroup $K \subset G$, there exists a finite index subgroup $J \subset Z(H)$ such that $J \subset C_{G}(K)$.

If $H \subset G$ is a ghost center, then it is always a weak ghost center. Indeed, given $K \subset G$ finitely generated, there exists a finite index $J \subset H$ such that $J \subset C_{G}(K)$. Then $J \cap Z(H)$ is finite index in $Z(H)$, and $J \cap Z(H) \subset J \subset C_{G}(K)$.

\begin{remark}
Note that, for trivial reasons, any \emph{finite} subgroup $H \subset G$ is automatically a ghost center.
\end{remark}

While our main application will be for $\mathbb{Z}^{d}$-SFT's, we will also prove a classification for weak ghost centers for certain expanded version of the stabilized automorphism groups over more general base groups.

\begin{definition}
For a $G$-subshift $X$, we define the $G$-augmented stabilized automorphism group by
$$\mathcal{J}(X,G) = \langle \autinfty(X,G), G \rangle \subset \Homeo(X).$$
\end{definition}

The finitary Ryan's theorems from the previous section allows us to classify finitely generated ghost centers in $\autinfty(X,\mathbb{Z}^{d})$ for contractible $\mathbb{Z}^{d}$ SFT's, and also in $\mathcal{J}(X)$ when $X$ is a full shift over an infinite finitely generated residually finite group $G$.

\begin{proposition}
Suppose that $G$ is an infinite finitely generated residually finite group and let $X$ be a potent full shift over $G$. Then a finitely generated subgroup $H \subset \mathcal{J}(X,G)$ is a weak ghost center if and only if $Z(H) \cap Z(G)$ is finite index in $Z(H)$.
\end{proposition}
\begin{proof}
First suppose $K = Z(H) \cap Z(G)$ is finite index in $Z(H)$ and let $E \subset \mathcal{J}(X,G)$ be a finitely generated subgroup. We claim that $E \subset \langle \aut(X,L),G \rangle$ for some finite index subgroup $L \subset G$. Indeed, if $E$ is generated by $g_{1},\ldots,g_{k}$, each of these generators can be written in terms of the generators of $G$ and finitely many elements $\alpha_{1},\ldots,\alpha_{i}$ from $\autinfty(X,G)$. Then there exists a finite index $L \subset G$ for which $\alpha_{1},\ldots,\alpha_{i} \in \aut(X,L)$, so $E \subset \langle \alpha_{1},\ldots,\alpha_{i},G \rangle \subset \langle \aut(X,L),G \rangle$. It follows that $L \cap C_{\scriptscriptstyle \mathcal{J}(X,G)}(G) \subset C_{\scriptscriptstyle \mathcal{J}(X,G)}(E)$. Now $C_{\scriptscriptstyle \mathcal{J}(X,G)}(G) \subset Z(G) \subset G$ since $G \subset \mathcal{J}(X,G)$, and $L$ is finite index in $G$, so $L \cap Z(G)$ is finite index in $Z(G)$. Since $K \subset Z(G)$, this implies $L \cap Z(G) \cap K$ is finite index in $K$ as well. But $L \cap Z(G) \subset C_{\scriptscriptstyle \mathcal{J}(X,G)}(E)$ so $L \cap Z(G) \cap K \subset C_{\scriptscriptstyle \mathcal{J}(X,G)}(E)$ as well. Finally, since $K$ was finite index in $Z(H)$ and $L \cap Z(G) \cap K$ is finite index in $K$, we have that $L \cap Z(G) \cap K$ is finite index in $Z(H)$, which shows that $H$ is a weak ghost center.

Now suppose $H \subset \mathcal{J}(X,G)$ is a weak ghost center. By Theorem~\ref{thm:potentsstrongryansomgsomany}, there exists a finitely generated subgroup $F \subset \aut(X,G)$ such that $C_{\scriptscriptstyle \mathcal{J}(X,G)}(F) = G$. Then by adding the generators of $G$ to $F$, we get a finitely generated subgroup $F^{\prime}$ such that $C_{\scriptscriptstyle \mathcal{J}(X,G)}(F^{\prime}) \subset C_{G}(G) = Z(G)$. Since $H$ is a weak ghost center, there exists a finite index subgroup $K \subset Z(H)$ such that $K \subset C_{\scriptscriptstyle \mathcal{J}(X,G)}(F^{\prime}) \subset Z(G)$. Thus $K \subset Z(H) \cap Z(G)$ and since $K$ is finite index in $Z(H)$ we have that $Z(H) \cap Z(G)$ is finite index in $Z(H)$.
\end{proof}
A similar argument classifies ghost centers for the stabilized automorphism groups of contractible $\mathbb{Z}^{d}$-SFT's.

\begin{proposition}\label{prop:ghostcenterchar}
Let $d \ge 1$ and $X$ be a contractible shift of finite type. Then a finitely generated subgroup $H$ is a ghost center in $\autinfty(X,\mathbb{Z}^{d})$ if and only if $H \cap \mathcal{Z}_{X}$ is finite index in $H$. If in addition $H$ is isomorphic to $\mathbb{Z}^{d}$, then $H$ is a ghost center if and only if $H$ is commensurable with $\mathcal{Z}_{X}$.
\end{proposition}
\begin{proof}
First suppose $K = H \cap \mathcal{Z}_{X}$ is finite index in $H$. Let $E$ be some finitely generated subgroup of $\autinfty(X,\mathbb{Z}^{d})$. Then $E \subset \aut(X,L)$ for some finite index subgroup $L \subset \mathbb{Z}^{d}$, and hence $E \subset C(L)$. Note that $L$, as a subgroup of $\mathcal{Z}_{X}$, is finite index in $\mathcal{Z}_{X}$. Since $K \subset \mathcal{Z}_{X}$, we have that $L \cap K$ is finite index in $K$ and hence also in $H$, since $K$ is finite index in $H$. Finally, $E \subset C(L) \subset C(L \cap K)$, so $L \cap K \subset C(E)$.

For the other direction, suppose now that $H$ is a ghost center in $\autinfty(X,\mathbb{Z}^{d})$ and is finitely generated. Then $H \subset \aut(X,J)$ for some finite index subgroup $J \subset \mathbb{Z}^{d}$. Since $X$ is contractible and hence has periodic points, by passing to a deeper finite index subgroup of $J$ if necessary, we can assume that $(X,J)$ has a fixed point (note that $(X,J)$ is still contractible, since if $(X,G)$ is contractible then $(X,H)$ is contractible for any subgroup $H \subset G$). Then by Theorem~\ref{thm:FinitaryRyan} there exists a finitely generated subgroup $E \subset \aut(X,J)$ such that the centralizer of $E$ in $\aut(X,J)$ (in fact, in $\textrm{Homeo}(X)$) is $\mathcal{Z}_{J}$. Since $H$ is a ghost center, there exists some finite index subgroup $L \subset H$ such that $L \subset C(E)$. It follows that $L \subset \mathcal{Z}_{J}$. Then $L \subset H \cap \mathcal{Z}_{X}$, and since $L$ is finite index in $H$, this implies $H \cap \mathcal{Z}_{X}$ is finite index in $H$.

The last part follows immediately, since if $H$ is isomorphic to $\mathbb{Z}^{d}$ and is virtually contained in $\mathcal{Z}_{X}$, then it is commensurable with $\mathcal{Z}_{X}$.
\end{proof}

\begin{remark}
It follows from the above proposition that when $X$ is contractible, any subgroup of $\mathcal{Z}_{X}$ in $\autinfty(X,\mathbb{Z}^{d})$ is a ghost center. In particular, while the proposition implies any ghost center in $\autinfty(X,\mathbb{Z}^{d})$ is virtually free abelian, it may be virtually $\mathbb{Z}^{k}$ for $k < d$.
\end{remark}

We can now show that any isomorphism of stabilized automorphism groups for contractible $\mathbb{Z}^{d}$-SFT's may be promoted to a pointed isomorphism relative to some finite index subgroups of the $\mathcal{Z}_{X}$'s.

\begin{theorem}\label{thm:isoupgrade}
Let $(X,\mathbb{Z}^{d}), (Y,\mathbb{Z}^{d})$ be contractible shifts of finite type and suppose there is an isomorphism of stabilized automorphism groups
$$\Psi \colon \autinfty(X,\mathbb{Z}^{d}) \to \autinfty(Y,\mathbb{Z}^{d}).$$
Then there are finite index subgroups $K_{1} \subset \mathcal{Z}_{X}, K_{2} \subset \mathcal{Z}_{Y}$ such that $\Psi$ is a pointed isomorphism
$$\Psi \colon \left(\autinfty(X,\mathbb{Z}^{d}),K_{1}\right) \to \left(\autinfty(Y,\mathbb{Z}^{d}),K_{2}\right).$$
\end{theorem}
\begin{proof}
By Proposition~\ref{prop:ghostcenterchar}, $\mathcal{Z}_{X}$ is a $\mathbb{Z}^{d}$-ghost center of $\autinfty(X,\mathbb{Z}^{d})$, and hence $\Psi(\mathcal{Z}_{X})$ is a $\mathbb{Z}^{d}$-ghost center of $\autinfty(Y,\mathbb{Z}^{d})$. Then Proposition~\ref{prop:ghostcenterchar} implies $\Psi(\mathcal{Z}_{X})$ is commensurable with $\mathcal{Z}_{Y}$, so $K = \Psi(\mathcal{Z}_{X}) \cap \mathcal{Z}_{Y}$ is finite index in $\Psi(\mathcal{Z}_{X})$ and $\mathcal{Z}_{Y}$. Setting $K_{1} = \Psi^{-1}(K)$ and $K_{2} = K$, then $K_{1}$ is finite index in $\mathcal{Z}_{X}$, $K_{2}$ is finite index in $\mathcal{Z}_{Y}$, and we have a pointed isomorphism
$$\Psi \colon \left( \autinfty(X,\mathbb{Z}^{d}),K_{1}\right) \to \left(\autinfty(Y,\mathbb{Z}^{d}),K_{2}\right).$$
\end{proof}

\section{Classification}
We are now in a position to show that topological entropy can be rationally recovered from the stabilized automorphism group for contractible $\mathbb{Z}^{d}$-SFT's. A classification of the stabilized automorphism groups of full shifts over $\mathbb{Z}^{d}$ will follow.

\begin{theorem}\label{thm:entropyrationalratio}
Suppose that $X$ and $Y$ are contractible $\mathbb{Z}^{d}$-shifts of finite type and that
$$\Psi \colon \autinfty(X,\mathbb{Z}^{d}) \to \autinfty(Y,\mathbb{Z}^{d})$$
is an isomorphism. Then $\frac{\htop(X)}{\htop(Y)} \in \mathbb{Q}$.
\end{theorem}
For the proof, naturally the goal is to use local $\mathcal{Q}$ entropy. However, the local $\mathcal{Q}$ entropy calculations given in Theorem~\ref{thm:fullshiftentropycalc} use ECS's taking input data from the specific subshift. Our first goal then is to show how to find adapted ECS's which are suitable for the local $\mathcal{Q}$ calculations for both systems simultaneously.

\sloppy Suppose then that $(X,\mathbb{Z}^{d})$ and $(Y,\mathbb{Z}^{d})$ are contractible $\mathbb{Z}^{d}$-SFT's and that {$\Psi \colon \autinfty(X,\mathbb{Z}^{d}) \to \autinfty(Y,\mathbb{Z}^{d})$} is an isomorphism. By Theorem~\ref{thm:isoupgrade}, there exists finite index subgroups $K_{1} \subset \mathcal{Z}_{X}, K_{2} \subset \mathcal{Z}_{Y}$ such that $\Psi(K_{1}) = K_{2}$. We may consider the SFT's $(X,K_{1})$ and $(Y,K_{2})$. Since both $(X,K_{1})$ and $(Y,K_{2})$ are still contractible,, both still possess the BEEPS property and have strong density of periodic points. But both $K_{1}$ and $K_{2}$ are finite index in $\mathcal{Z}_{X} \cong \mathbb{Z}^{d}$ and $\mathcal{Z}_{Y} \cong \mathbb{Z}^{d}$ respectively, so they are both abstractly isomorphic to $\mathbb{Z}^{d}$, and hence $(X,K_{1})$ and $(Y,K_{2})$ are $\mathbb{Z}^{d}$-SFT's with the BEEPS property.

Let $S$ be greater than the window size for both $(X,K_{1})$ and $(Y,K_{2})$ and let $f_{S}^{X},f_{S}^{Y}$ be BEEPS functions for $(X,K_{1})$ and $(Y,K_{2})$ respectively.
Both $(X,K_{1})$ and $(X,K_{2})$ may be assumed to be subshifts over some alphabets $\mathcal{A},\mathcal{B}$ respectively. We fix some $\kappa > \max\{|\mathcal{A}|,|\mathcal{B}|\}$.

Since $\Psi(K_{1}) = K_{2}$, the map $\Psi|_{K_{1}} \colon K_{1} \to K_{2}$ is an isomorphism. We denote this isomorphism by $\theta \colon K_{1} \to K_{2}$. Define functions
\begin{equation*}
\begin{gathered}
g_{S}^{X} \colon \fol(K_{1}) \times \mathbb{N} \to \mathbb{R}, \qquad g_{S}^{Y} \colon \fol(K_{2}) \times \mathbb{N} \to \mathbb{R}\\
g_{S}^{X}(\mathcal{F},m) = \max\{f_{S}^{X}(\mathcal{F},m),f_{S}^{Y}\left(\theta(\mathcal{F}),m)\right)\}\\
g_{S}^{Y}(\mathcal{F},m) = \max\{f_{S}^{Y}(\mathcal{F},m),f_{S}^{X}\left(\theta^{-1}(\mathcal{F}),m)\right)\}
\end{gathered}
\end{equation*}

\begin{lemma}
Both $g_{S}^{X}$ and $g_{S}^{Y}$ are controlled functions.
\end{lemma}
\begin{proof}
Given a F\o{}lner sequence $\mathcal{F} \in \fol(K_{1})$, the sequence $\theta(\mathcal{F})$ belongs to $\fol(K_{2})$. Then for every $m \in \mathbb{N}$ we have
$$\log g_{S}^{X}(\mathcal{F},m) \le \log \max\{f_{S}^{X}(\mathcal{F},m),f_{S}^{Y}(\theta(\mathcal{F}),m)\}.$$
Since both $f_{S}^{X}$ and $f_{S}^{Y}$ are controlled functions, it follows that
$$\frac{\log g_{S}^{X}(\mathcal{F},m)}{|F_{m}|} \to 0$$
so $g_{S}^{X}$ is controlled. The argument for $g_{S}^{Y}$ is analogous.
\end{proof}
\sloppy Let $\lambda \ge \max\{||\theta||_{\infty},||\theta^{-1}||_{\infty}\} \ge 1$.
Recall that, for the isomorphism ${\theta \colon K_{1} \to K_{2}}$, if $\mathcal{Q}$ is an ECS for $L$ then $\theta(\mathcal{Q})$ is an ECS for $K_{2}$.
\begin{lemma}\label{lemma:proofecscomparison}
$\mathcal{Q}(g_{S}^{Y},v_{\kappa,S}) \subset \theta(\mathcal{Q}(g_{S}^{X},v_{\kappa,\lambda S}))$ and $\mathcal{Q}(g_{S}^{X},v_{\kappa,S}) \subset \theta^{-1}(\mathcal{Q}(g_{S}^{Y},v_{\kappa,\lambda S}))$.
\end{lemma}
\begin{proof}
Let $\mathcal{F} \in \fol(K_{2})$ and $m \in \mathbb{N}$. First note that
$$v_{\kappa,S}(\mathcal{F},m) = v_{\kappa,S}(F_{m}) = \kappa^{|\partial_{S}F_{m}|}$$
$$ = \kappa^{|\theta^{-1}(\partial_{S}(F_{m}))|} \le \kappa^{|\partial_{\lambda S}\theta^{-1}(F_{m})|} $$
$$= v_{\kappa,\lambda S}(\theta^{-1}(F_{m})) = v_{\kappa,\lambda S}(\theta^{-1}(\mathcal{F}),m).$$
Thus
\begin{equation}\label{eqn:kappasandstuff}
v_{\kappa,S}(\mathcal{F},m) \le v_{\kappa,\lambda S}(\theta^{-1}(\mathcal{F}),m).
\end{equation}
Then we have
$$\theta(\mathcal{Q}(g_{S}^{X},v_{\kappa,\lambda S}))(\mathcal{F},m) = \mathcal{Q}(g_{S}^{X},v_{\kappa,\lambda S})(\theta^{-1}(\mathcal{F}),m)$$
$$=\mathcal{P}^{s}\left(g_{S}^{X}(\theta^{-1}(\mathcal{F}),m),v_{\kappa,\lambda S}(\theta^{-1}(\mathcal{F}),m)\right)$$
$$= \mathcal{P}^{s}\left(\max\{f_{S}^{X}(\theta^{-1}(\mathcal{F}),m),f_{S}^{Y}(\mathcal{F},m)\},v_{\kappa,\lambda S}(\theta^{-1}(\mathcal{F}),m)\right).$$
By~\eqref{eqn:kappasandstuff} we have $v_{\kappa,S}(\mathcal{F},m) \le v_{\kappa,\lambda S}(\theta^{-1}(\mathcal{F}),m)$
so this last class contains
$$\mathcal{P}^{s}\left(\max\{f_{S}^{X}(\theta^{-1}(\mathcal{F}),m),f_{S}^{Y}(\mathcal{F},m)\},v_{\kappa,S}(\mathcal{F},m)\right)$$
$$=\mathcal{Q}(g_{S}^{Y},v_{\kappa,S})(\mathcal{F},m)$$
and hence
$$\mathcal{Q}(g_{S}^{Y},v_{\kappa,S}) \subset \theta(\mathcal{Q}(g_{S}^{X},v_{\kappa,\lambda S})).$$

The argument for the second containment is analogous.
\end{proof}

Since $f^{X}(\mathcal{F},m) \le g^{X}(\mathcal{F},m)$ and $f^{Y}(\mathcal{F}^{\prime},m) \le g^{Y}(\mathcal{F}^{\prime},m)$ for any F\o{}lner sequences $\mathcal{F} \in \fol(K_{1}), \mathcal{F}^{\prime} \in \fol(K_{2})$, it follows that both $g^{X}$ and $g^{Y}$ are $S$-BEEPS functions for $X$ and $Y$ respectively.

Now $\Psi \colon \left(\autinfty(X,K_{1}),K_{1}\right) \to \left(\autinfty(Y,K_{2}),K_{2}\right)$ is a pointed isomorphism and $\Psi|_{K_{1}} \colon K_{1} \to K_{2}$ is given by $\theta$. The previous lemma then implies $\mathcal{Q}(g^{Y},v_{\kappa,S}) \subset \Psi(\mathcal{Q}(g^{X},v_{\kappa,\lambda S}))$ and  $\mathcal{Q}(g_{S}^{X},v_{\kappa,S}) \subset \Psi^{-1}(\mathcal{Q}(g_{S}^{Y},v_{\kappa,\lambda S}))$.

Since both $g_{S}^{X}$ and $g^{Y}_{S}$ are $S$-BEEPS functions for $X$ and $Y$ respectively, $\lambda S \ge S$ is greater than the window size for both $X$ and $Y$, and $\kappa$ is larger than both alphabets on which $X$ and $Y$ are defined, Theorem~\ref{thm:fullshiftentropycalc} implies
$$h_{\mathcal{Q}(g_{S}^{X},v_{\lambda S,\kappa})}(\autinfty(X),K_{1}) = \htop(X,K_{1}),$$
$$h_{\mathcal{Q}(g_{S}^{Y},v_{S,\kappa})}(\autinfty(Y),K_{2}) = \htop(X,K_{2}).$$

Then this, together with $\mathcal{Q}(g^{Y},v_{\kappa,S}) \subset \Psi(\mathcal{Q}(g^{X},v_{\kappa,\lambda S}))$ and Proposition~\ref{prop:pentropyproperties} implies
$$[\mathbb{Z}^{d} \colon K_{1}] \cdot \htop(X) = \htop(X,K_{1}) = h_{\mathcal{Q}(g_{S}^{X},v_{\lambda S,\kappa})}(\autinfty(X),K_{1})$$
$$ = h_{\Psi(\mathcal{Q}(g_{S}^{X},v_{\lambda S,\kappa}))}(\autinfty(Y),K_{2})$$
$$\ge  h_{\mathcal{Q}(g^{Y},v_{\kappa,S})}(\autinfty(Y),K_{2}) = \htop(Y,K_{2}) = [\mathbb{Z}^{d} \colon K_{2}] \cdot \htop(Y).$$
Thus
$$[\mathbb{Z}^{d} \colon K_{1}] \cdot \htop(X) \ge [\mathbb{Z}^{d} \colon K_{2}] \cdot \htop(Y).$$
Likewise, using that $\mathcal{Q}(g_{S}^{X},v_{\kappa,S}) \subset \Psi^{-1}(\mathcal{Q}(g_{S}^{Y},v_{\kappa,\lambda S}))$ and again Proposition~\ref{prop:pentropyproperties} we have
$$[\mathbb{Z}^{d} \colon K_{2}] \cdot \htop(Y) = \htop(Y,K_{2}) = h_{\mathcal{Q}(g_{S}^{Y},v_{\lambda S,\kappa})}(\autinfty(Y),K_{2})$$
$$ = h_{\Psi^{-1}(\mathcal{Q}(g_{S}^{Y},v_{\lambda S,\kappa}))}(\autinfty(X),K_{1})$$
$$\ge  h_{\mathcal{Q}(g_{S}^{X},v_{\kappa,S})}(\autinfty(X),K_{1}) = \htop(X,K_{1}) = [\mathbb{Z}^{d} \colon K_{2}] \cdot \htop(X)$$
so
$$[\mathbb{Z}^{d} \colon K_{2}] \cdot \htop(Y) \ge [\mathbb{Z}^{d} \colon K_{1}] \cdot \htop(X).$$
It follows altogether that
$$[\mathbb{Z}^{d} \colon K_{1}] \cdot \htop(X) = [\mathbb{Z}^{d} \colon K_{2}] \cdot \htop(Y)$$
and hence
$$\frac{\htop(X)}{\htop(Y)} = \frac{[\mathbb{Z}^{d} \colon K_{2}]}{[\mathbb{Z}^{d} \colon K_{1}]} \in \mathbb{Q}.$$

\begin{theorem}\label{thm:classification1}
Let $\mathcal{A}$ and $\mathcal{B}$ be finite alphabets of size at least two and let $d \ge 1$. Then $\autinfty(\mathcal{A}^{\mathbb{Z}^{d}},\mathbb{Z}^{d})$ and $\autinfty(\mathcal{B}^{\mathbb{Z}^{d}},\mathbb{Z}^{d})$ are isomorphic if and only if there exists $m,n \in \mathbb{N}$ such that $|\mathcal{A}|^{m} = |\mathcal{B}|^{n}$.
\end{theorem}
\begin{proof}
First suppose there exists $m,n \in \mathbb{N}$ such that $|\mathcal{A}|^{m} = |\mathcal{B}|^{n}$. Let $L_{1}$ be a subgroup of $\mathbb{Z}^{d}$ of index $m$ and $L_{2}$ a subgroup of $\mathbb{Z}^{d}$ of index $n$. Then $(\mathcal{A}^{\mathbb{Z}^{d}},L_{1})$ is topologically conjugate to $(\mathcal{A}^{[\mathbb{Z}^{d} \colon L_{1}]},L_{1}) = (\mathcal{A}^{m},L_{1})$. Likewise, $(\mathcal{B}^{\mathbb{Z}^{d}},L_{2})$ is topologically conjugate to $(\mathcal{B}^{n},L_{2})$. Since $|\mathcal{A}|^{m} = |\mathcal{B}|^{n}$, the systems $(\mathcal{A}^{m},L_{1})$ and $(\mathcal{B}^{n},L_{2})$ are topologically conjugate, so by Propositions~\ref{prop:someproperties1} and~\ref{prop:prop1}, we have
$$\autinfty(\mathcal{A}^{\mathbb{Z}^{d}},\mathbb{Z}^{d}) = \autinfty(\mathcal{A}^{\mathbb{Z}^{d}},L_{1}) \cong \autinfty(\mathcal{A}^{m},L_{1})$$
$$ \cong \autinfty(\mathcal{B}^{n},L_{2}) \cong \autinfty(\mathcal{B}^{\mathbb{Z}^{d}},L_{2}) = \autinfty(\mathcal{B}^{\mathbb{Z}^{d}},\mathbb{Z}^{d}).$$

For the other direction, suppose that $\autinfty(\mathcal{A}^{\mathbb{Z}^{d}},\mathbb{Z}^{d})$ and $\autinfty(\mathcal{B}^{\mathbb{Z}^{d}},\mathbb{Z}^{d})$ are isomorphic. By Theorem~\ref{thm:entropyrationalratio}, we have
$$\frac{\log|\mathcal{A}|}{\log |\mathcal{B}|} = \frac{\htop(\mathcal{A}^{\mathbb{Z}^{d}})}{\htop(\mathcal{B}^{\mathbb{Z}^{d}})} \in \mathbb{Q}$$
from which it follows that $|\mathcal{A}|^{q} = |\mathcal{B}|^{p}$ for some $p,q$.
\end{proof}

We also distinguish the dimension of the underlying $\mathbb{Z}^{d}$-action.
\begin{theorem}
Let $d_{1}, d_{2} \in \mathbb{N}$ and $\mathcal{A}, \mathcal{B}$ be finite alphabets of size at least two. The groups $\autinfty(\mathcal{A}^{\mathbb{Z}^{d_{1}}},\mathbb{Z}^{d_{1}})$ and $\autinfty(\mathcal{B}^{\mathbb{Z}^{d_{2}}},\mathbb{Z}^{d_{2}})$ are isomorphic if and only if $d_{1} = d_{2}$ and there exists $m,n$ such that $|\mathcal{A}|^{m} = |\mathcal{B}|^{n}$.
\end{theorem}
\begin{proof}
If $d_{1} = d_{2}$ and there exists such $m,n$ then the statement follows from Theorem~\ref{thm:classification1}. Suppose then that $\Psi \colon \autinfty(\mathcal{A}^{\mathbb{Z}^{d_{1}}},\mathbb{Z}^{d_{1}}) \stackrel{\cong}\longrightarrow \autinfty(\mathcal{B}^{\mathbb{Z}^{d_{2}}},\mathbb{Z}^{d_{2}})$ is an isomorphism. By Proposition~\ref{prop:ghostcenterchar}, $\mathcal{Z}_{\mathbb{Z}^{d_{1}}}$ is a ghost center in $\autinfty(\mathcal{A}^{\mathbb{Z}^{d_{1}}},\mathbb{Z}^{d_{1}})$, so $\Psi(\mathcal{Z}_{\mathbb{Z}^{d_{1}}})$ is a ghost center in $\autinfty(\mathcal{B}^{\mathbb{Z}^{d_{2}}},\mathbb{Z}^{d_{2}})$. Again by Proposition~\ref{prop:ghostcenterchar}, this implies $\Psi(\mathcal{Z}_{\mathbb{Z}^{d_{1}}})$ is commensurable with $\mathcal{Z}_{\mathbb{Z}^{d_{2}}}$, so $d_{1} = d_{2}$, and the result then follows again from Theorem~\ref{thm:classification1}.
\end{proof}

Finally, we discuss some examples of contractible shifts of finite type over $\mathbb{Z}^{d}$ whose stabilized automorphism groups are not isomorphic to the stabilized automorphism group of any full shift.

By Theorem~\ref{thm:entropyrationalratio}, it suffices to find $\mathbb{Z}^{d}$-SFT's which are contractible and whose topological entropy is not a rational multiple of $\log n$ for any $n \in \mathbb{N}$.\\

\emph{Example: } Let $X$ be a mixing $\mathbb{Z}$-SFT whose topological entropy is not a rational multiple of $\log n$ for any $n \in \mathbb{N}$; there are numerous such systems, for instance the golden mean shift. Choose some embedding $\mathbb{Z} \subset \mathbb{Z}^{d}$, and consider the \emph{free extension} of $(X,\mathbb{Z})$, which is the  $\mathbb{Z}^{d}$-subshift $(Y,\mathbb{Z}^{d})$ whose configurations consist of independently chosen configurations from $X$ on each coset of $\mathbb{Z}$ in $\mathbb{Z}^{d}$. It is straightforward to check that $\htop(Y,\mathbb{Z}^{d}) = \htop(X,\mathbb{Z})$. It is proved in~\cite[Lemma 8.10]{PoirierSaloContractible} that if $X$ is contractible, then any such free extension $Y$ is contractible. Since $(X,\mathbb{Z})$ is a mixing $\mathbb{Z}$-SFT, by~\cite[Theorem 1.3]{PoirierSaloContractible}, $(X,\mathbb{Z})$ is contractible, and hence $(Y,\mathbb{Z}^{d})$ is contractible as well.\\

\emph{Example: } For an explicit example, consider the Triangular Hard-Square system $(Z,\mathbb{Z}^{2})$, which is the  $\mathbb{Z}^{2}$-SFT defined on $\{0,1\}$ where horizontally, vertically and for ONE diagonal, 1s cannot be adjacent. Calculations~\cite{Baxter1980,Joyce1988} show that $\htop(Z,\mathbb{Z}^{2})$ is not a rational multiple of $\log n$ for any $n$, and hence $\aut^{(\infty)}(Z,\mathbb{Z}^{2})$ is not isomorphic to the stabilized automorphism group of any full shift.\\

Lastly, in some sense there must exist `many' contractible $\mathbb{Z}^{d}$-SFT's whose topological entropy is not a rational multiple of $\log n$ for any $n$. By~\cite[Prop. 3.2]{PoirierSaloContractible}, the golden mean shift over $\mathbb{Z}^{d}$ is contractible. We do not know whether its topological entropy is rationally related to that of a full shift. But if it was, we may simply consider any contractible $\mathbb{Z}^{d}$-SFT $X$ whose topological entropy satisfies $\htop(X) = \log c$ where $c$ is not a root of a rational number. Then $\htop(Y \times X) = \htop(Y) + \htop(X)$ would not be rationally related to the entropy of a full shift.


\bibliographystyle{plain}
\bibliography{bibliography}

\end{document}